\documentclass[12 pt]{amsart}
\usepackage{amssymb,amsmath,amsthm,amstext,amsfonts}
\usepackage{amsmath,amstext,amsthm,amsfonts}
\usepackage[dvips]{graphicx}
\usepackage{epic,eepic}
\usepackage[dvips]{graphicx}
\usepackage{psfrag}
\usepackage{txfonts}
\usepackage{color}
\usepackage{mathrsfs}
\usepackage{upgreek}

\makeatletter \@addtoreset{equation}{section} \makeatother

\renewcommand\thetable{\thesection.\@arabic\c@table}

\theoremstyle{plain}
\newtheorem{maintheorem}{Theorem}

\newtheorem{maincorollary}{Corollary}

\newtheorem{theorem}{Theorem }[section]
\newtheorem{proposition}[theorem]{Proposition}
\newtheorem{lemma}[theorem]{Lemma}
\newtheorem{corollary}[theorem]{Corollary}

\theoremstyle{definition} \theoremstyle{remark}
\newtheorem{remark}[theorem]{Remark}
\newtheorem{example}[theorem]{Example}
\newtheorem{definition}[theorem]{Definition}

\DeclareMathAlphabet{\mathpzc}{OT1}{pzc}{m}{it}

\newcommand{\field}[1]{\mathbb{#1}}

\newcommand{\SL}{\text{SL}}
\newcommand{\RR}{\field{R}}

\newcommand{\diam}{\operatorname{diam}}

\renewcommand{\epsilon}{\varepsilon}

\begin{document}
\large

\title[Generic area-preserving reversible diffeomorphisms]{Generic area-preserving reversible diffeomorphisms}

\author[M. Bessa]{M\'ario Bessa}
\address{Departamento de Matem\'atica, Univ. da Beira Interior,
Rua Marqu\^es d'\'Avila e Bolama, 6201-001 Covilh\~a, Portugal}
\email{bessa@ubi.pt}

\author[M. Carvalho]{Maria Carvalho}
\address{Centro de Matemática da Univ. do Porto,
Rua do Campo Alegre, 687,
4169-007 Porto, Portugal}
\email{mpcarval@fc.up.pt}

\author[A. Rodrigues]{Alexandre Rodrigues}
\address{Departamento de Matem\'atica, Univ. do Porto,
Rua do Campo Alegre, 687,
4169-007 Porto, Portugal}
\email{alexandre.rodrigues@fc.up.pt}

\date{\today}

\maketitle

\begin{abstract}
Let $M$ be a surface and $R:M\rightarrow M$ an involution whose set of fixed points is a submanifold with dimension $1$ and such that $DR_x \in \SL(2,\mathbb{R})$ for all $x$. We will show that there is a residual subset of $C^1$ area-preserving $R$-reversible diffeomorphisms which are either Anosov or have zero Lyapunov exponents at almost every point.
\end{abstract}

\small
\tableofcontents
\normalsize
\medskip

{\small \noindent\emph{MSC2010:} primary 37D25, 37C80; secondary 37C05.\\\emph{keywords:} Dominated splitting; Lyapunov exponent; reversibility.\\}

\bigskip

\section{Introduction}\label{intro}

Let $M$ be a compact, connected, smooth Riemannian two-dimensional manifold without boundary and $\mu$ its normalized area. Denote by $\text{Diff}_{\mu}^{1}(M)$ the set of all area-preserving $C^1$-diffeomorphisms of $M$ endowed with the $C^1$-topology. A diffeomorphism $f:M \rightarrow M$ is Anosov if $M$ is a hyperbolic set for $f$. In \cite{N}, it was proved that a generic area-preserving diffeomorphism is either Anosov or the set of elliptic periodic points is dense in the surface. More recently \cite{M2, M3}, another $C^1$-generic dichotomy in this setting has been established. For $f \in \text{Diff}_{\mu}^{1}(M)$ and Lebesgue almost every $x \in M$, the upper \emph{Lyapunov exponent} at $x$ is given by
$$\lambda^+(f,x)=\lim_{n \rightarrow +\infty}  \,\log \|Df_x^n\|^{1/n}$$
and is non-negative. The main theorem of \cite{Bo} states that there is a $C^1$-residual subset of maps in $\text{Diff}_{\mu}^{1}(M)$ which are either Anosov or have upper Lyapunov exponent zero at Lebesgue almost every point. In this paper we address a similar question within the subspace of $\text{Diff}_{\mu}^{1}(M)$ which exhibit some symmetry. More precisely, let $R:M\rightarrow M$ be a diffeomorphism such that $R\circ R$ is the Identity of $M$ and denote by $\text{Diff}^{~1}_{\mu, R}(M)$ the subset of maps $f \in \text{Diff}_{\mu}^{1}(M)$, called $R$-reversible, such that $R$ conjugates $f$ and $f^{-1}$, that is,
$$R\circ f=f^{-1}\circ R.$$
The spaces $\text{Diff}^{~1}(M)$, $\text{Diff}^{~1}_{\mu}(M)$ and  $\text{Diff}^{~1}_{\mu, R}(M)$ are Baire \cite{Hi, Kur}. Under the additional two hypotheses
\begin{itemize}
\item[- ] $\forall\,\,\, x \in M$, one has $DR_x \in \SL(2, \mathbb{R})$
\item[- ] $Fix(R):=\{x \in M: R(x)=x\}$ is a submanifold of $M$ with dimension equal to $1$,
\end{itemize}
our main result states that:

\begin{maintheorem}\label{BochiMane}
There exists a $C^1$-residual $\mathscr{R}_R\subset\text{Diff}^{~1}_{\mu, R}(M)$ whose elements either are Anosov or have zero upper Lyapunov exponent at Lebesgue almost every point.
\end{maintheorem}

As the torus $\mathbb{T}^2=\mathbb{R}^2/\mathbb{Z}^2$ is the only surface that may support an Anosov diffeomorphism \cite{F2}, we readily deduce that:

\begin{maincorollary} If $M\neq \mathbb{T}^2$, a $C^1$ generic $f\in \,\text{Diff}^{~1}_{\mu, R}(M)$ has zero upper Lyapunov exponent Lebesgue almost everywhere.
\end{maincorollary}

A dynamical symmetry is a geometric invariant which plays an important role through several applications in Physics, from the Classical \cite{Birkhoff} and Quantum Mechanics \cite{Prigogine} to Thermodynamics \cite{Kumicak}. In this context we may essentially distinguish two types of natural symmetries: those which preserve orbits and those which invert them. Much attention has been paid to the former (see, for instance, \cite{FMN, Golubitsky II} and re\-ferences therein); the latter, called reversibility \cite{Arnold84, AS, BR}, is a feature that most prominently arises in Hamiltonian systems and has become a useful tool for the analysis of periodic orbits and homoclinic or heteroclinic cycles \cite{Devaney76}. The references \cite{LR} and \cite{RQ} present a thorough survey on reversible dynamical systems.

Another, more studied, dynamical invariant by smooth maps is a symplectic form \cite{MZ}. Our research will be focused on surfaces, where symplectic maps are the area preserving ones. Some dynamical systems are twofold invariant, both reversible and symplectic, as the Chirikov-Taylor standard map \cite{Meiss} defined by
$$f_\sigma(x,y)=\left(x+y-\frac{\sigma}{2\pi}\sin(2\pi x),y-\frac{\sigma}{2\pi}\sin(2\pi x)\right)$$
where $\sigma\in\mathbb{R}$, which accurately describes the dynamics of some magnetic field lines and is $R$-reversible for
$$R(x,y)=\left(-x, y-\frac{\sigma}{2\pi}\sin(2\pi x)\right).$$
They exhibit many interesting properties but, to the best of our knowledge, only a few systematic comparisons between these two settings have been investigated, as in \cite{Devaney76}, \cite{Sevryuk}.

\section{Framework}

A substantial amount of information about the geometry of the stable/unstable manifolds may be obtained from the presence of non-zero Lyapunov exponents and the exis\-tence of a dominated splitting. Hence, it is of primary importance to understand when one can avoid vanishing exponents or to evaluate their prevalence. Several successful strategies to characterize a generic dynamics are worth mentioning: \cite{BV2,Bo2} for volume-preserving diffeomorphisms, sympletic maps and linear cocycles in any dimension; \cite{Bessa,BessaR} for volume-preserving flows; \cite{BessaDias} for Hamiltonians with two degrees of freedom; \cite{Hertz} for diffeomorphisms acting in a three-dimensional manifold. In what follows, we will borrow ideas and techniques from these articles. To extend them to area-preserving reversible dynamical systems, the main difficulties are to handle with hyperbolic pieces which are not the entire manifold (see Sections \ref{Zehnder} and \ref{Bowen}) and to set up a program of $C^1$ small perturbations which keep invariant both the area-preserving character and the reversibility (details in Section~\ref{Franks}) and in the meantime collapse the expanding directions into the contracting ones in such an extent that the upper Lyapunov exponent diminishes (as done in Section~\ref{mainproposition}).

\section{Preliminaries}\label{preliminaries}

In this section, we will discuss some of the consequences of reversibility and summarize a few properties of Lyapunov exponents and dominated splittings.

\subsection{Reversibility}
Let $R$ be an involution and $f\in\text{Diff}^{~1}_{\mu, R}(M)$. Geometrically, reversibi\-lity means that, applying $R$ to an orbit of $f$, we get an orbit of $f^{-1}$, as shown in Figure~\ref{Geometry1}.

\begin{figure}\begin{center}
\includegraphics[height=6cm]{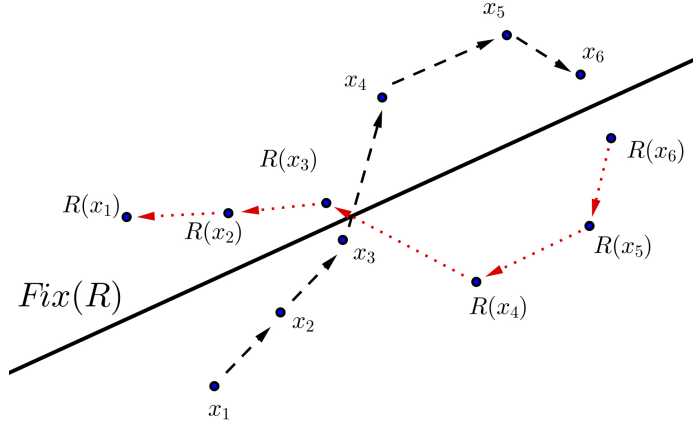}
\end{center}
\caption{Action of $R$.}
\label{Geometry1}
\end{figure}

The $f$-orbit of a point $x \in M$, say $\mathcal{O}(x)=\{f^n(x), x \in \mathbb{Z}\}$, is said to be $R$-symmetric if $R(\mathcal{O}(x))=\mathcal{O}(x)$. If $x$ is a fixed point by $f$ and its orbit is $R$-symmetric, then $x$ is a fixed point by $R$ as well. Yet, in general, the fixed point set of $f$ is not preserved by $R$.

Consider $f,g\in\text{Diff}^{~1}_{\mu, R}(M)$. Then $R\circ f^{-1}=f \circ R$ but $R\circ (f\circ g)=(f^{-1}\circ R)\circ g = (f^{-1}\circ g^{-1}) \circ R =(g \circ f)^{-1} \circ R$, so the set $\text{Diff}^{~1}_{\mu,R}(M)$ endowed with the composition of maps is, in general, not a group. Moreover, if $f\in\text{Diff}^{~1}_{\mu, R}(M)$ is conjugate through $h$ to $g\in\text{Diff}^{~1}_{\mu}(M)$, then, although $(R\circ h)\circ g=f^{-1}(R\circ h)$, g may be not $R$-reversible.

\medskip

\subsection{Dominated splitting}

In the sequel we will use the canonical norm of a bounded linear map $A$ given by $\|A\|=\sup_{\|v\|=1}\|A\, v\|$. For $f\in\text{Diff}^{~1}(M)$, a compact $f$-invariant set $\Lambda\subseteq{M}$ is said to be \emph{uniformly hyperbolic} if there is $m\in\mathbb{N}$ such that, for every $x\in\Lambda$, there is a $Df$-invariant continuous splitting $T_{x}M=E^{u}_{x}\oplus{E^{s}_{x}}$ such that
$$\|Df^{m}_x|_{{E}^{s}_{x}}\|\leq \frac{1}{2}\,\,\,\text{ and}\,\,\, \|(Df^{m}_x)^{-1}|_{{E}^{u}_{x}}\|\leq \frac{1}{2}.$$
There are several interesting ways to weaken the definition of uniform hyperbolicity. Here we use the one introduced in \cite{Liao, M1, Pl}. Given $m \in \mathbb{N}$, a compact $f$-invariant set $\Lambda\subseteq{M}$ is said to have a $m$-\emph{dominated splitting} if, for every $x\in\Lambda$, there exists a $Df$-invariant continuous splitting $T_x \Lambda=E_x^{u}\oplus{E_x^{s}}$ satisfying
\begin{equation}\label{ds}
\|Df^{m}_x|_{{E}^{s}_{x}}\|\,\,\|(Df^{m}_x)^{-1}|_{{E}^{u}_{x}}\|^{-1}\leq \frac{1}{2}.
\end{equation}
Observe that, if $\Lambda$ displays an $m$-dominated splitting for $f$, then the same splitting is domi\-nated for $f^{-1}$. Under a dominating splitting, both sub-bundles may expand or contract, although $E^u$ expands more efficiently than $E^s$ and, if both sub-bundles contract, $E^u$ is less contracting than $E^s$. Moreover, as happens in the uniform hyperbolicity setting, the angle between these sub-bundles is uniformly bounded away from zero, because the splitting varies continuously with the point and $\Lambda$ is compact, and the dominated splitting extends to the closure of $\Lambda$ (see~\cite{BDV} for full details). Within two-dimensional area-preserving diffeomorphisms, hyperbolicity is in fact equivalent to the existence of a dominated splitting \cite[Lemma 3.11]{Bo}.

\begin{lemma}
Consider $f\in\text{Diff}^{~1}_{\mu,R}(M)$ and a closed $f$-invariant set $\Lambda\subseteq M$ with a $m$-dominated splitting. Then $R(\Lambda)$ is closed, $f$-invariant and has a $m$-dominated splitting as well.
\end{lemma}

\begin{proof}
Clearly, as $R$ is an involution, $fR(\Lambda)=Rf^{-1}(\Lambda)=R(\Lambda)$. Let $x\in\Lambda$ whose orbit exhibits the decomposition $T_{f^i(x)} M=E^u_{f^i(x)}\oplus E^s_{f^i(x)}$, for $i\in\mathbb{Z}$. Then we also have a $Df$-invariant decomposition for $R(x)$, namely $T_{f^i(R(x))} M=E^u_{f^i(R(x))}\oplus E^s_{f^i(R(x))}$, for $i\in\mathbb{Z}$, where $$E^u_{f^i(R(x))}:=DR_{f^{-i}(x)}(E^s_{f^{-i}(x)})$$
and
$$E^s_{f^i(R(x))}:=DR_{f^{-i}(x)}(E^u_{f^{-i}(x)}).$$
Indeed, for $x\in\Lambda$ and $i\in\mathbb{Z}$, we have
\begin{eqnarray*}
Df_{f^i(R(x))}(E^s_{f^i(R(x))})&=&Df_{R(f^{-i}(x))}(E^s_{f^i(R(x))})=Df_{R(f^{-i}(x))} DR_{f^{-i}(x)}(E^u_{f^{-i}(x)})\\
&=&DR_{f^{-i-1}(x)} Df^{-1}_{f^{-i}(x)} (E^u_{f^{-i}(x)})=DR_{f^{-i-1}(x)}(E^u_{f^{-i-1}(x)})= E^s_{f^{i+1}(R(x))},
\end{eqnarray*}
and a similar invariance holds for the sub-bundle $E^u$. Therefore, since $R$ is a diffeomorphism in the compact $\Lambda$, we deduce that the angle between the sub-bundles at $R(x)$ is bounded away from zero. Finally, notice that
\begin{eqnarray*}
\|Df^{m}_{R(x)}|_{{E}^{s}_{R(x)}}\|\,\,\|(Df^{m}_{R(x)})^{-1}|_{{E}^{u}_{R(x)}}\|^{-1}&=&\|Df^{m}_{R(x)}|_{DR_x({E}^{u}_{x})}\|\,\,\|(Df^{m}_{R(x)})^{-1}|_{DR_x({E}^{s}_{x})}\|^{-1}\\
&=&\|R(Df^{m}_x)^{-1}|_{{E}^{u}_{x}}\|\,\,\|R(Df^{m}_x|_{{E}^{s}_{x}})\|^{-1}\\
&=&\|(Df^{m}_x)^{-1}|_{{E}^{u}_{x}}\|\,\,\|Df^{m}_x|_{{E}^{s}_{x}}\|^{-1} \overset{(\ref{ds})}{\leq} \frac{1}{2}.
\end{eqnarray*}
\end{proof}

\subsection{Lyapunov exponents}\label{LExp}

By the Oseledets' theorem \cite{O}, for $\mu$-a.e. point $x\in{M}$, there is a splitting
$T_{x}M=E^{1}_{x}\oplus...\oplus{E^{k(x)}_{x}}$ (called \emph{Oseledets' splitting}) and real numbers $\lambda_{1}(x)>...>\lambda_{k(x)}(x)$
(called \emph{Lyapunov exponents}) such that $Df_{x}(E^{i}_{x})=E^{i}_{f(x)}$ and
$$\underset{n\rightarrow{\pm\infty}}{\lim}\,\frac{1}{n}\log\|Df^{n}_x(v^j)\|=\lambda_j(f,x)$$
for any $v^{j}\in{E^{j}_{x}\backslash\{\vec{0}\}}$ and $j=1,...,k(x)$.
This allows us to conclude that, for $\mu$-a.e. $x$,
\begin{equation}\label{angle}
\lim_{n\rightarrow{\pm{\infty}}}\,\frac{1}{n}\log{|\det(Df^{n}_{x})|=\sum_{j=1}^{k(x)}\lambda_{j}(x)\dim(E^{j}_{x})},
\end{equation}
which is related to the sub-exponential decrease of the angle between any two subspaces of the Oseledets splitting along $\mu$-a.e. orbit.

Since, in the area-preserving case, we have $|\det(Df^{n}_{x})|=1$ for any $x \in M$, by~(\ref{angle}) we get
$$\lambda_{1}(x)+\lambda_{2}(x)=0.$$
Hence either
$\lambda_{1}(x)=-\lambda_{2}(x)>0$ or they are both equal to zero. If the former holds for $\mu$-a.e. $x$, then there are two one-dimensional subspaces $E^{u}_{x}$ and
$E^{s}_{x}$, associated to the positive Lyapunov exponent $\lambda_{1}(x)=\lambda_{u}(x)$ and the negative $\lambda_{2}(x)=\lambda_{s}(x)$, respectively. We denote by $\mathscr{O}(f)$ the set of regular points, that is,
$$\mathscr{O}(f)= \{x\in M: \lambda_{1}(x),\lambda_{2}(x)\,\, \text{ exist}\}$$
by $\mathscr{O}^{+}(f)\subseteq{\mathscr{O}(f)}$ the subset of points with one positive Lyapunov exponent
$$\mathscr{O}^{+}(f)= \{x\in \mathscr{O}(f): \lambda_{1}(x)>0\}$$
and by $\mathscr{O}^{0}(f)\subseteq{\mathscr{O}(f)}$ the set of those points with both Lyapunov exponents equal to zero
$$\mathscr{O}^{0}(f)= \{x\in \mathscr{O}(f): \lambda_{1}(x)=\lambda_{2}(x)=0\}.$$
So $\mathscr{O}^{+}(f)=\mathscr{O}(f)\backslash\mathscr{O}^{0}(f)$.

In this notation, we may summarize Oseledets theorem in the area-preserving reversible setting as:

\begin{theorem}(\cite{O})\label{Oseledec}
Let $f\in\text{Diff}^{~1}_{\mu,R}(M)$. For Lebesgue almost every $x\in{M}$, the limit
$$\lambda^{+}(f,x)=\underset{n\rightarrow{+\infty}}{\lim}\,\frac{1}{n}\log\|Df^{n}_x\|$$
exists and defines a non-negative measurable function of $x$. For almost any $x\in{\mathscr{O}^{+}}$, there is a splitting $E_{x}=E_{x}^{u}\oplus{E_{x}^{s}}$ which varies measurably with $x$ and satisfies:
$${v}\in{E_{x}^{u}}\backslash\{\vec{0}\} \,\, \Rightarrow  \,\,\underset{n\rightarrow{\pm\infty}}{\lim}\,\frac{1}{n}\log\|Df^{n}_x(v)\|=\lambda^{+}(f,x).$$
$${v}\in{E_{x}^{s}}\backslash\{\vec{0}\}\,\, \Rightarrow \,\,\underset{n\rightarrow{\pm\infty}}{\lim}\,\frac{1}{n}\log\|Df^{n}_x(v)\|=-\lambda^{+}(f,x).$$
$$\vec{0}\not={v}\notin{E_{x}^{u}}\,\cup\,E_{x}^{s}\,\,\Rightarrow \,\, \underset{n\rightarrow{+\infty}}{\lim}\,\frac{1}{n}\log\|Df^{n}_x(v)\|=\lambda^{+}(f,x)\,\text{ and } \,\underset{n\rightarrow{-\infty}}{\lim}\,\frac{1}{n}\log\|Df^{n}_x(v)\|=-\lambda^{+}(f,x).$$
\end{theorem}

The next result informs about a natural rigidity on the Lyapunov exponents of reversible diffeomorphisms.

\begin{lemma}
\label{Lema_importante}
Let $f\in\text{Diff}^{~1}_{\mu,R}(M)$. If $x\in \mathscr{O}^+$ has a decomposition $E^u_x\oplus E^s_x$ , then
\begin{itemize}
\item[(a)] $R(x)\in \mathscr{O}^+$.
\item[(b)] The Oseledets splitting at $R(x)$ is $E^u_{R(x)}\oplus E^s_{R(x)}$ with $E^u_{R(x)}=DR_x(E^s_x)$, $E^s_{R(x)}=DR_x(E^u_x)$.
\item[(c)] $\lambda^+(f,R(x))=\lambda^+(f,x)$ and $\lambda^-(f,R(x))=\lambda^-(f,x)=-\lambda^+(f,x)$.
\item[(d)] If $x\in \mathscr{O}^0$, then $R(x)\in \mathscr{O}^0$.
\end{itemize}
\end{lemma}

\begin{proof}
Assume that $x\in \mathscr{O}^+$ and let ${v}\in{E_{x}^{u}}\backslash\{\vec{0}\}$. Consider the direction $v':=DR_x(v)\in T_{R(x)} M$ and let us compute the Lyapunov exponent at $R(x)$ along this direction:
\begin{eqnarray*}
\lambda(f,R(x),v')&=&\underset{n\rightarrow{\pm\infty}}{\lim}\,\frac{1}{n}\log\|Df^{n}_{R(x)}(v')\|=\underset{n\rightarrow{\pm\infty}}{\lim}\,\frac{1}{n}\log\|D(R\circ f^{-n}\circ R)_{R(x)}(v')\|\\
&=&\underset{n\rightarrow{\pm\infty}}{\lim}\,\frac{1}{n}\log\|DR_{f^{-n}(R^2(x))} Df^{-n}_{R^2(x)} DR_{R(x)} DR_x(v)\|\\
&=&\underset{n\rightarrow{\pm\infty}}{\lim}\,\frac{1}{n}\log\|DR_{f^{-n}(x)} Df^{-n}_{x} (v)\|=\underset{n\rightarrow{\pm\infty}}{\lim}\,\frac{1}{n}\log\|Df^{-n}_{x} (v)\|\\
&=&-\underset{n\rightarrow{\pm\infty}}{\lim}\,\frac{1}{-n}\log\|Df^{-n}_{x}(v)\|= -\lambda^+(f,x,v).
\end{eqnarray*}
Thus $R(x)\in \mathscr{O}^+$. The other properties are deduced similarly.
\end{proof}

\subsection{Integrated Lyapunov exponent}\label{ILExp}

It was proved in \cite{Bo} that, when $\text{Diff}^{~1}_{\mu}(M)$ is endowed the $C^1$-topology and $[0,+\infty[$ has the usual distance, then the function
$$\begin{array}{cccc}
\mathscr{L}\colon & \text{Diff}^{~1}_{\mu}(M) & \longrightarrow  & [0,+\infty[ \\
& f & \longrightarrow  & \int_M \lambda^+(f,x)\,d\mu
\end{array}$$
is upper semicontinuous. This is due to the fact that $\mathscr{L}$ is the infimum of continuous functions, namely
\begin{equation}\label{inf}
\mathscr{L}(f)=\underset{n\in\mathbb{N}}{\inf}\,\,\frac{1}{n}\int_M \log\|Df^n_x\|d\mu.
\end{equation}
Clearly, the same holds for the restriction of $\mathscr{L}$ to $\text{Diff}^{~1}_{\mu,R}(M)$. Therefore, there exists a residual set in $\text{Diff}^{~1}_{\mu, R}(M)$ for which the map $\mathscr{L}$ is continuous \cite{Kur}. Now, the upper semicontinuity of $\mathscr{L}$ implies that $\mathscr{L}^{-1}(\left[0,\tau \right[)$ is $C^1$-open for any $\tau > 0$; hence
$$\mathscr{A}_{\tau}:=\left\{f\in\text{Diff}^{~1}_{\mu,R}(M)\colon \mathscr{L}(f)<\tau\right\}$$
is $C^1$-open.

\subsection{$(R,f)$-free orbits}

Given a subset $X$ of $M$, we say that $X$ is $(R,f)$\emph{-free} if
$$f(x)\not=R(y)\,\,\,\,\,\forall\,\, x,y \in X.$$

\begin{lemma}\label{free2}
Let $f\in \text{Diff}^{~1}_{\mu, R}(M)$. If $x\in M$ and $R(x)$ does not belong to the $f$-orbit of $x$, then this orbit is $(R,f)$-free.
\end{lemma}

\begin{proof} Let us assume that there exist $i,j\in\mathbb{Z}$ such that $f^i(x)=R(f^j(x))$. Then $f^i(x)=f^{-j}(R(x))$ and $f^{j+i}(x)=R(x)$, which contradicts the assumption. \end{proof}

\begin{proposition} \label{D}
There is a residual $\mathscr{D}\subset \text{Diff}^{~1}_{\mu, R}(M)$ such that, for any $f\in \mathscr{D}$, the set of orbits outside $Fix(R)$ which are not $(R,f)$-free is countable.
\end{proposition}

\begin{proof} Since $f$ and $R$ are smooth maps defined on $M$, by Thom transversality theorem \cite{Guillemin} there exists an open and dense set $\mathscr{D}_1\subset \text{Diff}^{~1}_{\mu, R}(M)$ such that, if $f \in \mathscr{D}_1$, the graphs of $f$ and $R$ are transverse submanifolds of $M\times M$, intersecting only at isolated points. Therefore, we may find a neighborhood of each intersection point where it is unique. By compactness of $M$, we conclude that generically the graphs of $f$ and $R$ intersect at a finite number of points (and this is an open property). Denote by $\mathcal{F}_1=\{x_j\}_{j=1}^{k_1}$ the set of points such that $f(x_j)=R(x_j)$.

Analogously, for $n\in \mathbb{N}$, let $\mathscr{D}_n\subset \text{Diff}^{~1}_{\mu, R}(M)$ be the open and dense set of diffeomorphisms $f \in \text{Diff}^{~1}_{\mu, R}(M)$ such that the graphs of $\{f^{-n},...,f^{-1},f,f^2,...,f^n\}$ and $R$ are transverse, and denote by $\mathcal{F}_n=\{x_j\}_{j=1}^{k_n}$ the finite set of $f$-orbits satisfying $f^i(x_j)=R(x_j)$ for some $j\in\{1,...,k_n\}$ and $i\in\{-n,...,-1,1,...,n\}$.

Finally, define
$$\mathscr{D}:=\bigcap_{n\,\in\,\mathbb{N}}\mathscr{D}_n$$
and the countable set of $(R,f)$-not-free orbits by
$$\mathcal{F}:=\bigcup_{n\,\in\, \mathbb{N},\,m\,\in\,\mathbb{Z}}f^m(\mathcal{F}_n).$$
We are left to show that, if $f\in \mathscr{D}$ and $x\in M\backslash\displaystyle [\mathcal{F}\cup Fix(R)]$, then the orbit of $x$ is a $(R,f)$-free set. Indeed, by construction, for such an $x$, the iterate $R(x)$ does not belong to the $f$-orbit of $x$; thus, by Lemma~\ref{free2}, this orbit is $(R,f)$-free.
\end{proof}

\begin{remark} The previous argument may be performed in $\text{Diff}^{~r}_{\mu, R}(M)$, for any $r\in \mathbb{N}$.
\end{remark}

From the previous result and the fact that $\dim\,(Fix(R))=1$, we easily get:

\begin{corollary}\label{D2}
Generically in $ \text{Diff}^{~1}_{\mu, R}(M)$, the set of $(R,f)$-free orbits has full measure.
\end{corollary}

\section{Linear examples on $M = \mathbb{T}^2$}\label{Anosov maps}

In this section we will address several questions concerning reversibility of linear Anosov diffeomorphisms, that is, difeomorphisms induced on $\mathbb{T}^2$ by the projection of a linear hyperbolic $L \in \SL(2,\mathbb{Z})$, with respect to linear involutions.

\subsection{Linear involutions}

We start characterizing the linear involutions $R:\mathbb{T}^2\rightarrow \mathbb{T}^2$ of the torus, induced by matrices $A$ in $\SL(2,\mathbb{Z})$. After differentiating the equality $R^2=Id_{\mathbb{T}^2}$ at any point of $\mathbb{T}^2$, we obtain $A^2=Id_{\mathbb{R}^2}$. Comparing the entries of the matrices in this equality, we conclude that:

\begin{proposition}\label{classification}
If $A\colon \mathbb{R}^2\rightarrow \mathbb{R}^2$ is a non trivial (that is, $A\not=\pm Id$) linear involution of $\SL(2,\mathbb{Z})$, then it belongs to the following list:\\
\begin{itemize}
\item $A=\begin{pmatrix}\pm1 & 0 \\ \gamma & \mp1 \end{pmatrix}$ \hspace{1cm} or its transpose, for some $\gamma\in\mathbb{Z}$.
\item $A=\begin{pmatrix} \alpha & \beta \\ \frac{1-\alpha^2}{\beta} & -\alpha \end{pmatrix}$ \hspace{0.8cm} for $\alpha,\beta\in\mathbb{Z}\backslash\{0\}$ such that $\beta$ divides $1-\alpha^2$.
\end{itemize}
\end{proposition}

\begin{proof}
Let $A$ be a matrix $\begin{pmatrix}\alpha & \beta \\ \gamma & \delta \end{pmatrix} \in \SL(2,\mathbb{Z})$ such that $A^2=Id$. This means that
$$\alpha \delta - \beta \gamma = \pm 1$$
and
\begin{equation*}
\left\{
\begin{array}{l}
\alpha^2+\beta\gamma=1\\
\beta(\alpha + \delta) = 0\\
\gamma(\alpha+\delta)=0\\
\gamma\beta+\delta^2=1
\end{array}
\right.
\end{equation*}
which implies that
\begin{equation*}
\left\{
\begin{array}{l}
\beta=0 \,\vee\, \alpha=-\delta\\
\gamma=0 \,\vee\, \alpha=-\delta.
\end{array}
\right.
\end{equation*}

\medskip

\noindent \textbf{1st case:} $\beta=0$

\bigskip
One must have $\alpha=\pm 1$ and $\delta=\pm 1$. If $\alpha=\delta=1$ or $\alpha=\delta=-1$, we conclude that $\gamma=0$ and so $A=\pm Id$. Therefore, $-\alpha=\delta=1$ or $\alpha=-\delta=1$, and there are no restrictions on the value of $\gamma$. Hence $A = \begin{pmatrix}\pm1 & 0 \\ \gamma & \mp1 \end{pmatrix}$, for $\gamma\in\mathbb{Z}.$

\bigskip

\noindent \textbf{2nd case:} $\gamma=0$

\bigskip

Again $\alpha=\pm 1$ and $\delta=\pm 1$, and so either $-\alpha=\delta=1$ or $\alpha=-\delta=1$. Therefore $A = \begin{pmatrix}\pm 1 & \beta \\ 0 & \mp1 \end{pmatrix}$, for $\beta\in\mathbb{Z}.$

\bigskip

\noindent \textbf{3rd case:} $\beta \neq 0$ and $\gamma \neq 0$

\bigskip
We must have $\alpha=-\delta$ and so $\alpha^2+\beta\gamma=1$ is equivalent to $\gamma=\frac{1-\alpha^2}{\beta}$. Moreover, $\beta$ must divide $1-\alpha^2$. Thus
$$A = \begin{pmatrix}\alpha & \beta \\ \frac{1-\alpha^2}{\beta} & -\alpha \end{pmatrix}.$$
\end{proof}

\bigskip

From this description and by solving the equation $A(x,y)=(x,y)$ in $\mathbb{R}^2$, we deduce that:

\begin{corollary}
The fixed point set of a linear non-trivial involution $R$ of the torus is a closed smooth curve.
\end{corollary}

\begin{proof} $R$ is induced by a matrix $A \in \SL(2,\mathbb{Z})$ of the previous list and

\medskip

\begin{center}
\small{\begin{tabular}{|c|c|}
	\hline
$A$ & Fixed point subspace of $A$ \\
	\hline\hline
$\begin{pmatrix} 1 & 0 \\ \gamma & - 1 \end{pmatrix}$ & $y=\frac{\gamma}{2}x$  \\
	\hline
$\begin{pmatrix} -1 & 0 \\ \gamma & 1 \end{pmatrix}$ & $x=0$  \\
	\hline
$\begin{pmatrix} \alpha & \beta \\ \frac{1-\alpha^2}{\beta} & -\alpha \end{pmatrix}, \,\,\beta \neq 0$ & $y=\frac{1-\alpha}{\beta}x $  \\
\hline
\end{tabular}}
\end{center}
\end{proof}

\subsection{Linear reversibility}

Let $R$ be a linear involution of the torus, induced by a matrix $A \in \SL(2,\mathbb{Z})$. Is there an $R$-reversible linear (area-preserving) Anosov diffeomorphism $f$? The answer is obvious (and no) if $R=\pm Id$. For the other possible involutions, we will look for a linear Anosov $f$ whose derivative at any point of $f$ is a fixed linear map with matrix $L\in \SL(2,\mathbb{Z})$, and, to simplify our task, we also assume that $det(L)=1$.

If we lift the equality $R\circ f=f^{-1}\circ R$, by differentiating it at any point of $\mathbb{T}^2$, we obtain $A\circ L=L^{-1}\circ A.$ Analyzing the entries of these matrices we conclude that:

\begin{proposition}\label{FindL}
Let $\mathcal{L}$ be the set of linear Anosov diffeomorphisms on the torus. If $R$ is a non-trivial (that is, $R \neq \pm Id$) linear involution, then
$$\text{Diff}^{~1}_{\mu, R}(\mathbb{T}^2) \cap \mathcal{L} \neq \emptyset.$$
\end{proposition}

\begin{proof}
Going through the possible matrices $A$ given by Proposition~\ref{classification}, we will determine, for each $R$, an orientation-preserving, linear, $R$-reversible Anosov diffeomorphism $f$, induced by a matrix $L(x,y)=(ax+by,cx+dy) \in \SL(2,\mathbb{Z})$. The entries of $L$ must satisfy the conditions:
\begin{itemize}
\item[(IL)] (Integer lattice invariance) $\,\,a,b,c,d\in\mathbb{Z}$ and $ad-bc=\pm1$.
\item[(H1)] (Hyperbolicity) $\,\,(a+d)^2-4>0$, if $ad-bc=1$.
\item[(H2)] (Hyperbolicity) $\,\,(a+d)^2+4$ is not a perfect square, if $ad-bc=-1$.
\end{itemize}

\medskip

\noindent Notice that (IL) ensures conservativeness and that the two properties (IL) and (H1 or H2) together imply that $b$ and $c$ do not vanish: otherwise $ad=\pm 1$ and this contradicts hyperbolicity. Moreover, if $ad-bc=1$, then $D=(a+d)^2 - 4$ is not a perfect square. This explains why a linear Anosov diffeomorphism is not $\pm Id$-reversible.\\

\noindent \textbf{1.} $A=\begin{pmatrix} 1 & 0 \\ \gamma & - 1 \end{pmatrix}$

\bigskip
	
The equality $A\circ L=L^{-1}\circ A$, with $det(L)=1$, is equivalent to $b\gamma=d-a$. If $\gamma=0$, we may take $L=\begin{pmatrix} a & b \\ c & a \end{pmatrix}$ with integer entries such that $a^2-bc=1$ and $4a^2>4$ (so $b \neq 0$ and $c \neq 0$). For instance, $a=d=3$, $b=4$ and $c=2$. If $\gamma \neq 0$, it must divide $d-a$ and $L$ has to be $\begin{pmatrix}a & \frac{d-a}{\gamma} \\ c & d \end{pmatrix}$, with integer entries such that $ad-bc=1$, $(a+d)^2>4$, $d-a \neq 0$ and $c\neq 0$. For example, $a=\gamma \in \mathbb{Z}\backslash \{0\}$, $b=1$, $c=2\gamma^2-1$ and $d=2\gamma$.

As the reversibility condition $A\circ L=L^{-1}\circ A$, with $det(L)=1$, is equivalent to $A^T\circ L^T=(L^T)^{-1}\circ A^T$, with $det(L^T)=1$, the case of the transpose matrix is equally solved.\\

\noindent \textbf{2.} $A=\begin{pmatrix}-1 & 0 \\ \gamma & 1 \end{pmatrix}$

\bigskip

As in the previous case, the equality $A\circ L=L^{-1}\circ A$, with $det(L)=1$, is equivalent to $b\gamma=a-d$. So, if $\gamma=0$, we may take $L=\begin{pmatrix}3 & 4 \\ 2 & 3 \end{pmatrix}$. If $\gamma \neq 0$, we may choose, for instance, $a=\gamma \in \mathbb{Z}\backslash \{0\}$, $b=-1$, $c=1-2\gamma^2$ and $d=2\gamma$. Again, for $A^T= \begin{pmatrix}-1 & \beta \\ 0 & 1 \end{pmatrix}$, we may just pick the Anosov diffeomorphism induced by $L^T$. \\

\bigskip

\noindent \textbf{3.} $A=\begin{pmatrix} \alpha & \beta \\ \frac{1-\alpha^2}{\beta} & -\alpha \end{pmatrix}$, with $\alpha, \beta, 1-\alpha^2 \neq 0$ and $\beta$ a divisor of $1-\alpha^2$

\bigskip

The equality
$$\begin{pmatrix} \alpha & \beta \\ \frac{1-\alpha^2}{\beta} &  -\alpha \end{pmatrix} \begin{pmatrix}a & b \\ c & d \end{pmatrix}=\begin{pmatrix}d & -b \\ -c & a \end{pmatrix}\begin{pmatrix}\alpha & \beta \\ \frac{1-\alpha^2}{\beta} &  -\alpha \end{pmatrix}$$
is equivalent to the equation
$$\alpha a +\beta c= \alpha d - \frac{b}{\beta}(1-\alpha^2)$$
that is,
\begin{equation}
\label{equation3}
\alpha\beta a +(1-\alpha^2) b +\beta^2 c -\alpha \beta d=0.
\end{equation}
Now to simplify our task, let us try to find a matrix satisfying $a=d$. Under this assumption, equation (\ref{equation3}) becomes
$$(1-\alpha^2) b +\beta^2 c=0.$$
As $c$ must comply with the equality $a^2-bc=1$ and $b$ cannot be zero, we must have $c=\frac{a^2-1}{b}$. In addition to this, we know that $\beta$ divides $1-\alpha^2$ and that $\alpha^2\neq 1$, so $4\alpha^2 -4 > 0$. Therefore, a convenient choice is $a=d=\alpha$, $b=\pm\beta$ and $c=\frac{\alpha^2-1}{\pm\beta}$. This way,
$$L=\begin{pmatrix} \alpha & \beta \\ \frac{\alpha^2-1}{\beta} & \alpha \end{pmatrix} \,\,\,\,\text{ or }\,\,\,\, L=\begin{pmatrix} \alpha & -\beta \\ \frac{1-\alpha^2}{\beta} & \alpha \end{pmatrix}.$$
\end{proof}

\subsection{Linear reversible Anosov diffeomorphisms}

We will now discuss whether, given a linear Anosov diffeomorphism $f$, there are non-trivial linear involutions $R$ such that $f$ is $R$-reversible. In spite of the fact that, on the torus $\mathbb{T}^2$, each Anosov diffeomorphism is conjugate to a hyperbolic toral automorphism \cite{Man}, the conclusions we will draw cannot be extended to all the Anosov diffeomorphisms because $R$-reversibility, for a fixed $R$, is not preserved by conjugacy. Notice, however, that a diffeomorphism conjugate to a $R$-reversible linear Anosov diffeomorphism is reversible as well, although with respect to another involution which is conjugate to $R$ but, in general, not a diffeomorphism.

\subsubsection{Orientation-preserving case}

Let $f$ be a linear Anosov diffeomorphism, induced by a matrix $L=\begin{pmatrix} a & b \\ c & d \end{pmatrix}\in \SL(2,\mathbb{Z})$, and assume that $ad-bc=1$. Take a linear involution $R$, given by the projection on the torus of a matrix $A$ as described in Proposition~\ref{classification}.\\

\noindent \textbf{Case 1}: $A=\begin{pmatrix} 1 & 0 \\ \gamma & -1 \end{pmatrix}$ or $A=\begin{pmatrix} -1 & 0 \\ \gamma & 1 \end{pmatrix}$.

\bigskip

The reversibility equality is equivalent to $b\gamma = d-a$ or $b\gamma = a-d$. So there is such an involution $A$ if and only if $b$ divides $d-a$, in which case only one valid $\gamma$ exists (namely, $\gamma=\frac{d-a}{b}$ or $\gamma=\frac{a-d}{b}$, respectively).\\

\noindent \textbf{Case 2}: $A=\begin{pmatrix} 1 & \gamma \\ 0 & -1 \end{pmatrix}$ or $A=\begin{pmatrix} -1 & \gamma \\ 0 & 1 \end{pmatrix}$.

\bigskip

Dually, the reversibility condition is equivalent to $c\gamma = d-a$ or $c\gamma = a-d$. So there is such an involution $A$ if and only if $c$ divides $d-a$, and then we get a unique value for $\gamma$.\\

\noindent \textbf{Case 3}: $A=\begin{pmatrix} \alpha & \beta \\ \frac{1-\alpha^2}{\beta} & -\alpha \end{pmatrix}$, where $\alpha, \beta, 1-\alpha^2 \neq 0$ and $\beta$ divides $1-\alpha^2$.

\bigskip

The pairs $(\alpha,\beta)\in\mathbb{Z}^2$ for which $f$ is $R$-reversible are the integer solutions of the equation, in the variables $\alpha$ and $\beta$, given by
$$\alpha a + \beta c = \alpha d - \frac{b}{\beta}(1-\alpha^2)$$
that is,
$$b\alpha^2 + \alpha \beta (d-a) - \beta^2 c = b.$$
This quadratic form defines a non-degenerate (because $b\neq 0$) conic whose kind depends uniquely on the sign of
$$\Delta=(d-a)^2 + 4bc = (a+d)^2 - 4$$
which we know to be always positive. So the conic is a hyperbola. After the change of variables
$$x=2b\alpha+(d-a)\beta \,\,\,\,\,\, \text{ and } \,\,\,\,\,\,y=\beta$$
the equation of the conic becomes
$$x^2-Dy^2=N$$
where $D=\Delta=(a+d)^2-4$ and $N=4b^2$. Thus the problem of finding the intersections of the conic with the integer lattice is linked to the solutions of this generalized Pell equation (and we need solutions with $y \neq 0$). According to \cite{C,M,Mo}, this Pell equation has zero integer solutions or infinitely many, and there are several efficient algorithms\footnote{See, for instance, http://www.alpertron.com.ar/QUAD.HTM} to determine which one holds in each particular case. However, if they exist, the solutions have also to fulfill the other requirements, namely $\alpha, \beta \neq 0$ and $\beta$ divides $1-\alpha^2$.

\begin{example}$\,$
\bigskip
\begin{center}
\small{\begin{tabular}{|c|c|c|c|c|c|}
	\hline
 \textbf{\emph{Anosov}} & & & \textbf{\emph{Involutions}} & &  \\
\hline\hline
& $\begin{pmatrix} 1 & 0 \\ \gamma & -1 \end{pmatrix}$ & $\begin{pmatrix} 1 & \gamma \\ 0 & -1 \end{pmatrix}$ & $\begin{pmatrix} -1 & 0 \\ \gamma & 1 \end{pmatrix}$ &  $\begin{pmatrix} -1 & \gamma \\ 0 & 1 \end{pmatrix}$ & $\begin{pmatrix} \alpha & \beta \\ \frac{1-\alpha^2}{\beta} & -\alpha \end{pmatrix}$ \\
\hline\hline
$\begin{pmatrix} 2 & 1 \\ 3 & 2 \end{pmatrix}$ & $\gamma=0$ & $\gamma=0$ & $\gamma=0$ & $\gamma=0$ & Example: $\,\begin{pmatrix} 2 & 1 \\ -3 & -2 \end{pmatrix}$ \\
\hline
$\begin{pmatrix} 2 & 1 \\ 1 & 1 \end{pmatrix}$ & $\gamma=-1$ & $\gamma=-1$ & $\gamma=1$ & $\gamma=1$ & Example: $\,\begin{pmatrix} 5 & 3 \\ -8 & -5 \end{pmatrix}$ \\
\hline
$ \begin{pmatrix} 4 & 9 \\ 7 & 16 \end{pmatrix}$ & $-$ & $-$ & $-$ & $-$ & $-$ \\
	\hline
\end{tabular}}
\end{center}
\bigskip
\end{example}

For $L=\begin{pmatrix} 2 & 1 \\ 3 & 2 \end{pmatrix}$, the generalized Pell equation is $x^2-12y^2=4$ and there are infini\-tely many matrices $A$ of type $3$ which correspond to linear involutions $R$ such that $f$ is $R$-reversible. Similarly, for $L=\begin{pmatrix} 2 & 1 \\ 1 & 1 \end{pmatrix}$, the generalized Pell equation is $x^2-5y^2=4$ and there are infinitely many solutions of type $3$. The third example in this table, whose generalized Pell equation is $x^2-396y^2=324$ and has infinitely many solutions but $L$ has no linear involutions, has been previously mentioned in \cite{BR}.\\

Notice that, if $R$ is an involution such that $R \circ f= f^{-1} \circ R$, then, for each $n\in \mathbb{Z}$, the diffeomorphism $R \circ f^n$ is also an involution, since
$$(R \circ f^n)^2 = (R \circ f^{n}) \circ (f^{-n}\circ R) = Id$$
and $f$ is $(R \circ f^n)$-reversible, because
$$(R \circ f^n)\circ f= (R \circ f) \circ f^n=(f^{-1}\circ R) \circ f^n = f^{-1}\circ (R \circ f^n).$$
Therefore, once such an involution $R$ is found for an Anosov diffeomorphism $f$, then we have infinitely many involutions with respect to which $f$ is reversible: no non-trivial power of an Anosov diffeomorphism is equal to the Identity, so, for any $k\neq m \in \mathbb{Z}$, we have $R\circ f^k \neq R\circ f^m$.

\subsubsection{Orientation-reversing case}

Consider now a linear Anosov diffeomorphism $f$, induced by a matrix $L=\begin{pmatrix}a & b \\ c & d \end{pmatrix}\in \SL(2,\mathbb{Z})$ such that $ad-bc=-1$. The previous analysis extends to this setting with similar conclusions. Indeed:\\

\noindent \textbf{Cases 1,\,2}: There is no valid $A$, since reversibility demands that $b$, $c$ or $a+d$ is zero, a value forbidden in this context.

\bigskip

\noindent \textbf{Case 3}: $A=\begin{pmatrix} \alpha & \beta \\ \frac{1-\alpha^2}{\beta} & -\alpha \end{pmatrix}$, where $\alpha, \beta, 1-\alpha^2 \neq 0$ and $\beta$ divides $1-\alpha^2$.

\bigskip

The pairs $(\alpha,\beta)\in\mathbb{Z}^2$ for which $f$ is $R$-reversible form the set of integer solutions of the equations, in the variables $\alpha$ and $\beta$, given by
$$\left\{ \begin{array}{c}
\alpha b + \beta d = 0 \\
\alpha c - \frac{a}{\beta}(1-\alpha^2)=0 \\
\alpha a + \beta c = -\alpha d + \frac{b}{\beta}(1-\alpha^2).
\end{array}
\right.$$
The third equality describes a (possibly degenerate) conic
$$b\alpha^2 + \alpha \beta (a+d)+ \beta^2 c = b$$
whose sort is determined by the sign of
$$\Delta=(a+d)^2 -4bc = (a-d)^2 - 4.$$
For instance, $\Delta <0$ for $L=\begin{pmatrix} 2 & 3 \\ 1 & 1 \end{pmatrix}$; $\Delta = 0$ when $L=\begin{pmatrix}3 & 4 \\ 1 & 1 \end{pmatrix}$; and $\Delta >0$ if $L=\begin{pmatrix} 4 & 5 \\ 1 & 1 \end{pmatrix}$. Once again, the problem of finding the points of this conic in the integer lattice is linked to the generalized Pell equation $x^2-Dy^2=N$, where $x=2b\alpha+(a+d)\beta$, $D=\Delta=(a-d)^2-4$, $N=4b^2$ and $y=\beta$, and also to the existence of solutions of the Pell equation satisfying the other two constraints, namely $\alpha b + \beta d = 0$ and $\alpha c - \frac{a}{\beta}(1-\alpha^2)=0$. \\

\begin{example} $\,$
\bigskip
\begin{center}
\small{\begin{tabular}{|c|c|c|c|c|c|}
	\hline
 \textbf{\emph{Anosov}} & \textbf{$\Delta$}  &\emph{\textbf{Generalized Pell Equation}} & \textbf{\emph{Number of solutions}} & \textbf{\emph{Conic}} &\textbf{\emph{Involutions}}  \\
\hline\hline
$\begin{pmatrix} 2 & 3 \\ 1 & 1 \end{pmatrix}$ & $-3$ & $x^2+3y^2=36$ & $6$ & Ellipse & $-$  \\
\hline
$\begin{pmatrix} 3 & 4 \\ 1 & 1 \end{pmatrix}$ & $0$ & $x^2=64$ & $\infty$ & Two vertical lines & $-$ \\
\hline
$\begin{pmatrix} 4 & 5 \\ 1 & 1 \end{pmatrix}$ & $5$ & $x^2-5y^2=100$ & $\infty$ & Hyperbola & $-$ \\
	\hline
\end{tabular}}
\end{center}
\bigskip
\end{example}

\begin{proposition}
If $f$ is an orientation-reversing linear Anosov diffeomorphism, there are no linear involutions $R$ such that $f$ is $R$-reversible.
\end{proposition}

\begin{proof} The only case still open is of the matrices $A=\begin{pmatrix} \alpha & \beta \\ \frac{1-\alpha^2}{\beta} & -\alpha \end{pmatrix}$ with $\alpha, \beta \neq 0$. Let us go back to the three conditions arising from reversibility in this setting:\\
$$\left\{ \begin{array}{c}
\alpha b + \beta d = 0 \\
\alpha \beta c - a(1-\alpha^2)=0 \\
b\alpha^2 + \alpha \beta (a+d)+ \beta^2 c = b.
\end{array}
\right.$$

\medskip

\noindent Replacing on the third equality $\alpha b$ by $-\beta d$, we get
$$\alpha \beta a +\beta^2 c = b.$$
Then, multiplying this equation by $\alpha$, which is nonzero, and turning $\alpha \beta c$ into $a(1-\alpha^2)$, we arrive at $\beta a = \alpha b$. This, joined to $\alpha b =- \beta d$, yields $\beta(a+d)=0$. A $\beta \neq 0$, we must have $a+d=0$, a value banned by the Anosov diffeomorphism $f$.
\end{proof}

\section{Generic examples}\label{Anosov diffeomosphisms}

Given an area-preserving diffeomorphims $f$, the r-centralizer de $f$, we denote by $\mathcal{Z}_r(f)$, is the set of involutions $R$ such that $R \circ f = f^{-1} \circ R$. If $f^2=Id$, then $Id$ and all the powers of $f$ belong to $\mathcal{Z}_r(f)$; and conversely. However, the Kupka-Smale theorem for area-preserving diffeomorphisms \cite{Devaney76} asserts that, given $k\in \mathbb{N}$, $C^1$-generically the periodic orbits of period less or equal to $k$ are isolated. So a generic $f \in \text{Diff}^{~1}_{\mu}(M)$ does not satisfy the equality $f^n=Id$, for any integer $n\neq 0$. Moreover, if $R\neq S$ are in $\mathcal{Z}_r(f)$, then $R \circ S$ belongs to the centralizer of $f$, due to
$$(R \circ S) \circ f = R \circ (S \circ f) = R \circ (f^{-1} \circ S) = f \circ (R \circ S).$$
Now, according to \cite{BCW}, for a $C^1$-generic $f \in \text{Diff}^{~1}_{\mu}(M)$, the centralizer of $f$ is trivial, meaning that it reduces to the powers of $f$. Therefore, there exists $n \in \mathbb{Z}$ such that $S=R\circ f^n$. We will say that $\mathcal{Z}_r(f)$ is trivial if it is either empty or there is an involution $R\neq Id$ generating it in the sense just explained.

\begin{proposition}
$C^1$-generically, the r-centralizer of $f \in \text{Diff}^{~1}_{\mu}(M)$ is trivial.
\end{proposition}

\medskip

\section{Stability of periodic orbits}\label{Stability}

Let $R:M \rightarrow M$ be an involution such that $DR_x \in \SL(2, \mathbb{R})$, for each $x \in M$, and $Fix(R):=\{x \in M: R(x)=x\}$ is a submanifold of $M$ with dimension equal to $1$. Consider $f \in \text{Diff}^{~1}_{\mu,R}(M)$. For area-preserving diffeomorphisms, hyperbolicity is an open but not dense property. Indeed, the $C^1$-stable periodic points are hyperbolic or elliptic; furthermore, in addition to openness, the area-preserving diffeomorphisms whose periodic points are either elliptic or hyperbolic are generic \cite{R_1970}. A version of Kupka-Smale theorem for reversible area-preserving diffeomorphisms has been established in \cite{Devaney76}. It certifies that, for a generic $f$ in $\text{Diff}^{~1}_{\mu,R}(M)$, all the periodic orbits of $f$ with given period are isolated.

\begin{theorem}\cite{Devaney76}\label{KS1}
Let
$$\mathscr{S}_k:=\{f\in \text{Diff}^{~r}_{\mu,R}(M)\colon \text{ every periodic point of period $\leq k$ is elementary}\}$$
$$\mathscr{S}:=\bigcap_{k\,\in\, \mathbb{N}}\mathscr{S}_k.$$
Then, for each $k,\,r \in \mathbb{N}$, the set $\mathscr{S}_k$ is $C^r$-residual in $\text{Diff}^{~r}_{\mu,R}(M)$. Thus, $\mathscr{S}$ is also $C^r$-residual.
\end{theorem}

Therefore, a generic $f \in \text{Diff}^{~r}_{\mu,R}(M)$ has countably many periodic points.

\begin{corollary}\label{E}
There is a residual $\mathscr{E}_r\subset \text{Diff}^{~r}_{\mu, R}(M)$ such that, for any $f\in \mathscr{E}_r$, the set of periodic points of $f$ has Lebesgue measure zero.
\end{corollary}

In \cite{Webster}, the author states generic properties of reversible vector fields on $3$-dimensional manifolds. To convey those features to diffeomorphisms on surfaces, we take the vector field defined by suspension of a reversible diffeomorphism $f:M\rightarrow M$, without losing differentiability \cite{PT}, acting on a quotient manifold $\bar{M}=M\times \mathbb{R}/\sim$ where it is transversal to the section $M\times {0}/\sim$. This vector field is reversible with respect to the involution obtained by projecting $R\times (-Id)$, whose fixed point set is still a submanifold of dimension $1$ of $\bar{M}$. This way, we deduce from \cite{Webster} that:

\begin{proposition}\label{KS2}
A generic $f\in \text{Diff}^{~r}_{\mu,R}(M)$ has only asymmetric fixed points and all its periodic orbits are hyperbolic or elliptic.
\end{proposition}

\section{Local perturbations}\label{Franks}

Let $R:M \rightarrow M$ be an involution as in the previous section. Consider $f \in \text{Diff}^{~1}_{\mu,R}(M)$. If $p \in M$ and we differentiate the equality $R \circ f = f^{-1} \circ R$ at $p$, then we get $DR_{f(p)} \circ Df_p = Df_{R(p)}^{-1}\circ DR_p$, a linear constraint between four matrices of $\SL(2,\mathbb{R})$, two of which are also linked through the equality $R^2=Id$. As the dimension $\SL(2,\mathbb{R})$ is $3$, there is some room to perform non-trivial perturbations.

In this section, we set two perturbation schemes that are the ground of the following sections. The first one describes a local small $C^1$ perturbation within reversible area-preserving diffeomorphisms in order to change a map and its derivative at a point, provided $x$ has a $(R,f)$-free non-periodic orbit by $f$. The second one is inspired by Franks' Lemma (\cite{F}), proved for dissipative diffeomorphisms, and allows to perform locally small abstract perturbations, within the reversible setting, on the derivative along a segment of an orbit of an area-preserving reversible diffeomorphism. These perturbation lemmas have been proved in the $C^1$ topology only, for reasons appositely illustrated in \cite{PSa, BDV}.

\subsection{1st perturbation lemma}\label{Franks1}

Consider $f \in \text{Diff}^{~1}_{\mu,R}(M)$ and take a point $x\in M$ whose orbit by $f$ is not periodic and $f(x)\not= R(x)$. Notice that those points exist if $f\in\mathscr{D}_1\cap\mathscr{E}_1$, as described in Proposition~\ref{D} and Corollary~\ref{E}. We will see how to slightly change $f$ and $Df$ at a small neighborhood of $x$ without losing reversibility.

Denote by $B(x,\rho)$ the open ball centered at $x$ with radius $\rho$ and by $C$ the union $B(x,\rho) \cup R(f(B(x,\rho)))$.

\begin{lemma}\label{local}
Given $f \in \text{Diff}^{~1}_{\mu,R}(M)$ and $\eta>0$, there exist $\rho>0$ and $\zeta>0$ such that, for any point $x\in M$, whose orbit by $f$ is not periodic and $f(x)\not= R(x)$, and every $C^1$ area-preserving diffeomorphism $h\colon M\rightarrow M$, coinciding with the Identity in $M\backslash B(x,\rho)$ and $\zeta$-$C^1$-close to the Identity, there exists $g\in\text{Diff}^{~1}_{\mu, R}(M)$ which is $\eta$-$C^1$-close to $f$ and such that $g = f$ outside $C$ and $g=f\circ h$ in $B(x,\rho)$.
\end{lemma}

\begin{proof}
Using the uniform continuity of $f$ on the compact $M$ and the fact that $f$ is $C^1$, we may choose $\tau> 0$ such that, each time the distance between two points $z$ and $w$ of $M$ is smaller than $\tau$, then the distance between their images by $f$, the norm of the difference of the linear maps $Df_z$ and $Df_w$ and the norm of the difference of the linear maps $DR_z$ and $DR_w$ are smaller than $\min\,\left\{\frac{\eta}{2}, \frac{\eta}{2\,\|f\|_{C^1}\,\|R\|_{C^1}}\right\}$.

As $f(x)\not= R(x)$, calling on the continuity of both $f$ and $R$ we may find $0< \rho <\tau$ such that the open ball $B(x,\rho)$ satisfies $f(B(x,\rho))\cap R(B(x,\rho))=\emptyset$ (or, equivalently, $B(x,\rho)\cap R(f(B(x,\rho)))=\emptyset$) and $B(x,\rho)\cap f(B(x,\rho))=\emptyset$. (See Figure~\ref{Closing Lemma 1A}.)

\begin{figure}\begin{center}
\includegraphics[height=5cm]{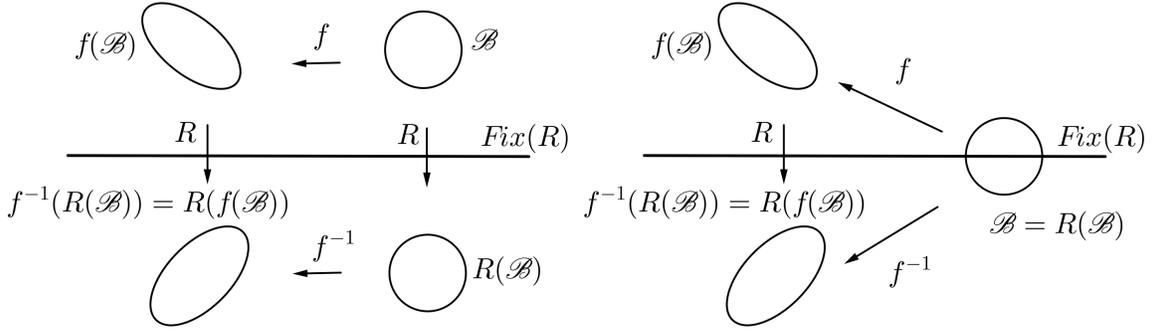}
\end{center}
\caption{Illustration of the 1st perturbation lemma: $B$ is the ball $B(x,r)$.}
\label{Closing Lemma 1A}
\end{figure}

Afterwards, consider
$$\zeta:=\frac{1}{2}\,\min\,\displaystyle\left\{\tau, \frac{\eta}{2\,\max\,\{\,\,\|f\|_{C^1}\,(\|R\|_{C^1})^2, \|f\|_{C^1}\,\,\}}\right\}$$
and take a $C^1$ area-preserving diffeomorphism $h\colon M\rightarrow M$ equal to the Identity in $M\backslash~B(x,\rho)$ and $\zeta$-$C^1$-close to the Identity. Finally, define $g:M \rightarrow M$ by

\begin{itemize}
\item $g=f\,\,$ outside $C$.
\item $g=f\circ h\,\,$ in $B(x,\rho)$.
\item $g= R\circ h^{-1} \circ f^{-1}\circ R\,\,$ in $R(f(B(x,\rho)))$.
\item $g= f\,\,$ in $R(B(x,\rho)) \cup f(B(x,\rho))$.
\end{itemize}

We begin by showing that the equality $R\circ g=g^{-1}\circ R$ holds. If $y\notin B(x,\rho)\cup f(B(x,\rho))$, then $R(y)$ is also out of this union and, therefore, $g(y)=f(y)$ and $g^{-1}(R(y))=f^{-1}(R(y))$. Hence $R(g(y))= R(f(y))=f^{-1}(R(y))= g^{-1}(R(y)).$ If $y\in B(x,\rho)$, then $R(y)\in R(B(x,\rho))$ and so
$$R(g(y))=R(f \circ h)(y)=R(f \circ h)(R\circ R)(y)=(R\circ h^{-1}\circ f^{-1} \circ R)^{-1}(R(y))=g^{-1}(R(y)).$$
Analogous computations prove the reversibility condition on $R(f(B(x,\rho)))$. Finally, if $y\in R(B(x,\rho))$, then $R(y)\in B(x,\rho)$ and $R(g(y))=R(f(y))= f^{-1}(R(y))= g^{-1}(R(y))$. Similar reasoning works for $y\in f(B(x,\rho))$.

Now we need to check that $g$ is $\eta$-$C^1$-close to $f$.\\

\noindent (a) $C^0$-approximation.\\

By definition, the differences between the values of $g$ and $f$ are bounded by the distortion the map $h$ induces on the ball $B(x,\rho)$ plus the effect that deformation creates on the first iterate by $f$ and the action of $R$ (which preserves distances locally). Now, for $z\in B(x,\rho)$, the distance between $h(z)$ and $z$ is small than $\zeta$, which is smaller than $\tau$. So, by the choice of $\tau$, the distance between $g(z)$ and $f(z)$ is smaller than $\eta$.\\

\noindent (b) $C^1$-approximation.\\

We have to estimate, for $z\in B(x,\rho)$, the norm $\|Df_z - Dg_z\|=\|Df_z - Df_{h(z)}(Dh_z)\|$ and, for $z \in R(f(B(x,\rho)))$, $\|Df_z - D(R\circ h^{-1}\circ f^{-1}\circ R)_z\|.$ Concerning the former, from the choices of $\tau$ and $\zeta$, we have
\begin{eqnarray*}
\|Df_z - Df_{h(z)} Dh_z\| &\leq& \|Df_z - Df_{h(z)}\| + \|Df_{h(z)}- Df_{h(z)} \,Dh_z\| \\
&\leq& \frac{\eta}{2}+ \|f\|_{C^1} \,\|Id_z - Dh_z\| \\
&\leq& \frac{\eta}{2}+ \|f\|_{C^1}\,\zeta < \eta.
\end{eqnarray*}
Regarding the latter,
\begin{eqnarray*}
& & \|Df_z - D(R\circ h^{-1}\circ R\circ f)_z\|= \\
&=& \|Df_z - D(R\circ h^{-1}\circ R)_{f(z)}\,Df_z\| \\
&\leq & \|Id_{f(z)} - D(R\circ h^{-1} \circ R)_{f(z)}\|\,\|f\|_{C^1}  \\
&=& \|DR_{R(f(z))}\,DR_{f(z)} - D(R\circ h^{-1})_{R(f(z))}\,DR_{f(z)}\|\,\|f\|_{C^1} \\
&\leq & \|DR_{R(f(z))} - D(R\circ h^{-1})_{R(f(z))}\|\,\|f\|_{C^1}\,\|R\|_{C^1} \\
&\leq & \|DR_{R(f(z))} - DR_{h^{-1}(R(f(z)))} \, Dh^{-1}_{R(f(z))}\|\,\|f\|_{C^1}\,\|R\|_{C^1} \\
&\leq & \frac{\eta}{2} + \|Id_{R(f(z))} - Dh^{-1}_{R(f(z))}\|\,\|f\|_{C^1}\,(\|R\|_{C^1})^2 \\
&\leq & \frac{\eta}{2} + \zeta \,\|f\|_{C^1}\,(\|R\|_{C^1})^2 < \eta.
\end{eqnarray*}
\end{proof}

As the set of Anosov area-preserving diffeomorphisms on the torus is $C^1$-open, combi\-ning the information of both Proposition~\ref{FindL} and Lemma~\ref{local} with the fact that an involution is not an Anosov, we conclude that:

\begin{corollary}
For any non-trivial linear involution $R$ on the torus, the (non-empty) space of area-preserving, $R$-reversible Anosov diffeomorphisms on the torus has no isolated points.
\end{corollary}

\subsection{2nd perturbation lemma}\label{Franks2}

We will now consider an area-preserving reversible diffeomorphism, a finite set in $M$ and an abstract tangent action that performs a small perturbation of the derivative along that set. Then we will search for an area-preserving reversible diffeomorphism, $C^1$ close to the initial one, whose derivative equals the perturbed cocycle on those iterates. To find such a perturbed diffeomorphism, we will benefit from the argument, suitable for area-preserving systems, presented in \cite{BDP}. But before proceeding, let us analyze the following example.

\begin{example}\label{noFranks}
Take the linear involution $R$ induced on the torus by the linear matrix $A(x,y)=(x,-y)$, and consider the diffeomorphism $f=R$. Clearly, $R\circ f=f^{-1}\circ R$. The set of fixed points of $f$, say $Fix(f)$, is the projection on the torus of $[0,1]\times\{0\} \cup [0,1] \times\{\frac{1}{2}\}$, and so it is made up by two closed curves. All the other orbits of $f$ are periodic with period $2$. Given $p\notin Fix(f)$, we have $Df_p=Df_{f(p)}=\begin{pmatrix}1 & 0 \\ 0 & -1 \end{pmatrix}$. Now, if $\eta>0$ and
$$L(p)=\begin{pmatrix}1+\eta & 0 \\ \eta & -\frac{1}{1+\eta} \end{pmatrix} \,\,\,\,\,\,\,\,\,\,\,\, L(f(p))=\begin{pmatrix}1+\eta & 0 \\ 0 & -\frac{1}{1+\eta} \end{pmatrix}$$
we claim that no diffeomorphism $g$ on the torus such that $Dg_p=L(p)$, $Dg_{f(p)}=L(f(p))$ and $g(p)=f(p)$ can be $R$-reversible. Indeed, differentiating the equality $R\circ g=g^{-1}\circ R$ at $p$, we would get
$$A\circ Dg_p=Dg^{-1}_{R(p)}\circ A= Dg^{-1}_{f(p)}\circ A$$
that is,
$$\begin{pmatrix}1 & 0 \\ 0 & -1 \end{pmatrix}\begin{pmatrix}1+\eta & 0 \\ \eta & -\frac{1}{1+\eta} \end{pmatrix}=\begin{pmatrix}\frac{1}{1+\eta} & 0 \\ \eta & -\frac{1}{1+\eta} \end{pmatrix}\begin{pmatrix}1 & 0 \\ 0 & -1 \end{pmatrix}$$
which would imply that
$\begin{pmatrix}1+\eta & 0 \\ -\eta & \frac{1}{1+\eta} \end{pmatrix}=\begin{pmatrix}\frac{1}{1+\eta} & 0 \\ \eta & \frac{1}{1+\eta}\end{pmatrix},$
an impossible equality when $\eta > 0$.

\end{example}

\medskip

This example evinces the need to impose some restrictions on the set where we wish to carry the perturbation.

\begin{lemma}\label{Franks3}
Fix an involution $R$ and $f\in\text{Diff}^{~1}_{\mu,R}(M)$. Let $X:=\{x_1,x_2,...,x_k\}$ be a finite $(R,f)$-free set of distinct points in $M$ whose orbits by $f$ are not periodic. Denote by $V=\oplus_{x\in X} T_x M$ and $V'=\oplus_{x\in X} T_{f(x)} M$ and let $P\colon V\rightarrow V'$ be a map such that, for each $x\in X$, $P(x) \in \SL(T_x M \rightarrow T_{f(x)} M)$. For every $\eta>0$, there is $\zeta>0$ such that, if $\|P-Df\|<\zeta$, then there exists $g\in\text{Diff}^{~1}_{\mu,R}(M)$ which is $\eta$-$C^1$-close to $f$ and satisfies $Dg_x=P|_{T_x M}$ for every $x\in X$. Moreover, if $K\subset M$ is compact and $K\cap X=\emptyset$, then $g$ can be found so that $g=f$ in $K$.
\end{lemma}

\begin{proof}
Given $\eta>0$, take the values of $\rho>0$ and $\zeta>0$ associated to $\frac{\eta}{k}$ by Lemma~\ref{local}, and note that each element of $X$ satisfies the hypothesis of this Lemma. Starting with $x_1$ and using Franks' Lemma for area-preserving diffeomorphisms \cite{BDP}, we perform a perturbation of $f$ supported in $B(x_1,\rho_1)$, where $0<\rho_1<\rho$ is sufficiently small, obtaining $G_1\in\text{Diff}^{~1}_{\mu}(M)$ such that $DG_{1_{x_1}}=P(x_1)$ and $G_1$ is $\zeta$-close to $f$.

Define $h_1=f^{-1}\circ G_1$. The $C^1$ diffeomorphism $h_1$ is area-preserving, equal to the Identity in $M\backslash B(x_1,\rho_1)$ and $\zeta$-$C^1$-close to the Identity. So, by Lemma~\ref{local}, there is $g_1\in\text{Diff}^{~1}_{\mu,R}(M)$ which is $\frac{\eta}{k}$-$C^1$-close to $f$, $g_1=f$ outside $C_1= B(x_1,\rho_1)\cup R(f(B(x_1,\rho_1)))$ and $g_1=f\circ h_1=G_1$ inside $B(x_1,\rho_1)$.

We proceed repeating the above argument for $x_2$ and $g_1$ just constructed, taking care to choose an open ball centered at $x_2$, with radius $0<\rho_2<\rho$, such that
$C_2=B(x_2,\rho_2)\cup R(f(B(x_2,\rho_2)))$ does not intersect $C_1$: this is a legitimate step according to the constraints $X$ has to fulfill. Applying again \cite{BDP}, we do a perturbation on $g_1$ supported in $B(x_2,\rho_2)$, which yields $G_2\in\text{Diff}^{~1}_{\mu}(M)$ such that $DG_{2_{x_2}}=P(x_2)$ and $G_2$ is $\zeta$-close to $g_1$. Therefore, the $C^1$ diffeomorphism $h_2=g_1^{-1}\circ G_2$ is area-preserving, equal to the Identity in $M\backslash B(x_2,\rho_2)$ and $\zeta$-$C^1$-close to the Identity. So, by Lemma~\ref{local}, there is $g_2\in\text{Diff}^{~1}_{\mu,R}(M)$ which is $\frac{\eta}{k}$-$C^1$-close to $g_1$, thus $\frac{2\eta}{k}$-$C^1$-close to $f$, satisfies $g_2=g_1$ outside $C_2$ and is such that $g_2=g_1\circ h_2=G_2$ inside $B(x_2,\rho_2)$.

In a similar way we do the remaining $k-2$ perturbations till we have taken into consideration all the elements of $X$. At the end of this process we obtain a diffeomorphism $g\in\text{Diff}^{~1}_{\mu,R}(M)$ which is $\eta$-$C^1$-close to $f$ and differs from $f$ only at $C=M\backslash\bigcup_{i=1}^k \, C_i$.\\

Surely, if $K$ is compact and $K\cap X=\emptyset$, then $C$ may be chosen inside the complement of $K$.

\end{proof}

\section{Smoothing out a reversible diffeomorphism}\label{Zehnder}

In this section, guided by \cite{Z}, we verify that a $C^1$ reversible diffeomorphism of the open and sense set $\mathscr{D}_1$ (see Proposition~\ref{D}) can be smoothed as a $R$-reversible $C^\infty$ diffeomorphism up to a set of arbitrarily small Lebesgue measure.

\begin{proposition}\cite{Z}\label{ZL_theorem}
Given $f\in \mathscr{D}_1$ and a pair of positive real numbers $\eta$ and $\epsilon$, there exist $g \in \text{Diff}^{~1}_{\mu, R}(M)$, which is $\eta$-$C^1$-close to $f$, and a compact $Z \subset M$ such that $\mu(M\backslash Z)<\epsilon$ and $g$ is $C^\infty$ in $Z$.
\end{proposition}

\begin{proof}
Assume $f$ is not $C^2$ and denote by $\mathcal{F}=\{x_i\}_{i=1}^k$ the set of elements of $M$ such that $f(x_i)=R(x_i)$. For arbitrary $\eta>0$ and $\epsilon>0$, take the open covering of $\mathcal{F}$ defined by $\bigcup_{i=1}^k B(x_i, r(\epsilon, \eta))$, denominate
$$\mathcal{B}= \bigcup_{i=1}^k B(x_i, r(\epsilon, \eta)) \cup R\left(\bigcup_{i=1}^k B(x_i, r(\epsilon, \eta)) \right)$$
and consider the compact set
$$Z=M\backslash \mathcal{B}$$
where $r(\epsilon, \eta)$ is chosen small enough to guarantee that $\mu(M\backslash Z)<\epsilon$. \\

Now select a finite open covering $\mathcal{U}_1:=\bigcup_{i=1}^\ell U_i$ such that $\mathcal{U}_1\cap \mathcal{F}=\emptyset$ and take the union $\mathcal{U}:=\mathcal{U}_1\cup R(f(\mathcal{U}_1)).$ By \cite{Z}, it is possible to smooth out the diffeomorphism $f$ in $\mathcal{U}$ by locally smoothing its generating functions associated to a selection of sympletic charts. Moreover, since we have $f(z)\neq R(z)$ for all $z \in \mathcal{U}$, we can perform a balanced perturbation, as explained in Lemma~\ref{local}, in order to ensure that the resulting diffeomorphism is $R$-reversible: each time we smooth in $U_i$, we also induce smoothness in $R(f(U_i))$.

The argument has a final recurrent step: the diffeomorphism $g$ is the limit, in the $C^{\infty}$ topology, of a sequence of $R$-reversible diffeomorphisms which are $C^{\infty}$ in $\mathcal{U}$ and $\eta$-$C^1$-close to $f$. As reversibility is a closed property, the limit $g$ is $R$-reversible too.
\end{proof}

\begin{remark}
If the previous argument is applied to $f \in \text{Diff}^{~1}_{\mu,R}(M)$ such that $f(x) \neq R(x)$ for all $x \in M$, then $Z=M$.
\end{remark}

\section{Hyperbolic sets}\label{Bowen}

It is well-known \cite{Bowen75} that \emph{basic} (non-Anosov) hyperbolic sets of $C^2$ diffeomorphisms have zero measure. In \cite{BV}, it was proved that the same result holds for compact hyperbolic sets without assuming that they are basic pieces. In what follows we will extend this property to the context of area-preserving reversible surface diffeomorphisms.

Given $f\in \mathscr{D}_1$ (see Proposition~\ref{D}) and positive real numbers $\eta$ and $\epsilon$, apply Proposition~\ref{ZL_theorem} to get a $\eta$-$C^1$-close-to-$f$ diffeomorphism $g$ in $\text{Diff}^{~1}_{\mu, R}(M)$ and a compact $Z \subset M$ such that $\mu(M\backslash Z)<\epsilon$ and $g$ is $C^\infty$ in $Z$. Adapting the the argument of \cite{BV}, we will show that the uniformly hyperbolic sets of $g$ have Lebesgue measure smaller than $\epsilon$, unless $g$ is Anosov.

\begin{proposition}\label{BV para reversiveis}
If $\Lambda$ is a compact hyperbolic set for $g$, then either $\mu(\Lambda)>0$ and $\mu(\Lambda\cap~Z)=0$, or $\mu(\Lambda)=0$, or else $\Lambda= M$.
\end{proposition}

\begin{proof}
Let $\Lambda$ be a compact hyperbolic set for $g$ and denote by $\hat\Lambda$ the closure $\overline{\Lambda\cap Z}$, to where hyperbolicity extends. We will prove that, if $\mu(\hat\Lambda)>0$, then $\Lambda=M$.

Recall that (see detailed information in \cite{HPS}), denoting by $d$ the induced Euclidean distance in $M$, for each $\rho >0$ and every $x \in \Lambda$, the local stable manifold of $x$ is defined as the subset
$$W^s_\rho(x)=\{y \in M: d(g^n(x), g^n(y))<\rho, \quad \forall n \in \mathbb{Z}_0^+ \} $$
and, similarly, the local unstable manifold of $x$ is
$$ W^u_\rho(x)=\{y \in M: d(g^n(x), g^n(y))<\rho, \quad  \forall n \in \mathbb{Z}_0^- \}.$$
As $g$ is $C^2$ in $\hat\Lambda$, the unstable foliation of points in $\hat\Lambda$ is absolutely continuous \cite{BP}. Let $\mu_u$ be the $u$-dimensional Lebesgue measure along the unstable one-dimensional mani\-folds of points in $\Lambda$. On account of $\mu(\hat\Lambda)>0$, there exists a density point $x_0$ of $\hat\Lambda$, that is, a point whose balls satisfy
$$\lim_{t \rightarrow 0}\,\, \frac{\mu(\hat\Lambda \cap B(x_0, t))}{\mu(B(x_0, t))}=1.$$
Therefore, if $t$ is small enough, then $\mu(\hat\Lambda \cap B(x_0, t))>0$ and we may find $y_0 \in \hat\Lambda \cap B(x_0, t)$ and $\rho>0$ such that $\mu_u (W^u_{\rho}(y_0) \cap \hat\Lambda \cap B(x_0, t))>0$. Hence there is $z_0 \in \hat\Lambda \cap B(x_0, t)$ which is a density point of $W^u_{\rho}(y_0) \cap \hat\Lambda \cap B(x_0, t)$ with respect to the measure $\mu_u$.

Since $\hat\Lambda$ is not invariant, the orbit of $z_0$ may move away from $Z$. To cope with these escapes, consider the sequence of returns of the $g$-orbit of $z_0$ to $\hat\Lambda \cap B(x_0, t)$, say $\left(z_{i}\right)_{i \in \mathbb{N}}=\left(g^{k_i}(z_0)\right)_{i \in \mathbb{N}}$, whose existence is ensured by Poincar\'e Recurrence Theorem \cite{KH}. For a fixed sufficiently small $\rho >0$, observe that
$$\lim_{i \rightarrow +\infty}\,\, \diam \left[g^{-{k_i}}(W^u_\rho(z_{i}))\right]=0 \,\,\quad \text{and}\quad\,\, \lim_{i \rightarrow +\infty} \,\,\frac{\mu_u(g^{-{k_i}}(W^u_\rho(z_{i})\backslash \hat\Lambda))}{\mu_u(g^{-{k_i}}(W^u_\rho(z_{i})))}=0,$$
where by $diam$ of a compact set $A \subset M$ we mean the maximum of the set $\{d(x,y): x,y \in A\}$. Using the bounded distortion of $C^2$ maps, we conclude that
$$\lim_{i \rightarrow +\infty}\,\, \frac{\mu_u(W^u_\rho(z_{i})\backslash \hat\Lambda)}{\mu_u(W^u_\rho(y_{i}))}=0.$$
What is more, as $0<{\mu_u(W^u_\rho(z_{i}))}<\epsilon$, we also have $\lim_{i \rightarrow +\infty} \,\,\mu_u(W^u_\rho(z_{i})\backslash \hat\Lambda)=0$. Take now a convergent subsequence of $\left(z_{i}\right)_{i \in \mathbb{N}}$ in the compact $\hat\Lambda$ and let $\ell \in \hat\Lambda$ be its limit. The disks $W^u_{\rho}(z_i)$ converge, as $i$ goes to $+\infty$, to $W^u_{\rho}(\ell)$ and therefore, by compactness of $\hat\Lambda$, we have $W^u_\rho(\ell)\subset \hat\Lambda$. Furthermore, if $U$ is an open small neighborhood of $\hat\Lambda$, $\hat\Lambda_U$ is the maximal invariant set of $g$ inside $U$ and $V_\ell$ is a closed neighborhood of $\ell$ contained in $U$, then we have $\mu(\hat\Lambda \cap V_\ell)>0$.
Since $g$ is area-preserving and $\mu(\hat\Lambda \cap V_\ell)>0$, by Poincar\'e Recurrence Theorem there exists $q \in \hat\Lambda \cap V_\ell$ and $n_0 \in \mathbb{N}$ such that $g^{n_0}(q) \in \hat\Lambda \cap V_\ell$. Applying the Shadowing Lemma \cite{KH}, we find a periodic point $p \in M$ of period $n_0$ such that, for all $j \in \{0, \ldots, n_0\}$, we have $g^j(p) \in U.$ The local invariant manifolds of $p \in M$ are close to those of $\ell$, thus $W^s_{loc}(p)$ intersects transversely $W^u_{loc}(\ell)$. By the $\lambda$-Lemma \cite{PM}, we conclude that $W^s_{loc}(\ell)$ $C^1$-accumulates at $W^u(p)$ and then, using the compactness of $\hat\Lambda$, we infer that $W^u (p)\subset \hat\Lambda$. In particular, $p \in \hat\Lambda$.

Define $Y=\overline{W^u(p)}$. The submanifolds $W^s_{loc}(Y)=\bigcup_{x\, \in\, Y} W^s_{loc}(x)$ and $W^u_{loc}(Y)=\bigcup_{x\, \in\, Y} W^u_{loc}(x)$ are open sets \cite{HPS} contained in a small neighborhood of $Y$. Moreover,

\begin{lemma}\cite[Lemma B7]{BV}
\begin{itemize}
\item[(a)] $W^u_{loc}(Y)=W^u_{loc}(x)$, for any $x \in Y$.\\

\item[(b)] $g\left(W^s_{loc}(Y)\right)= W^s_{loc}(Y)$.
\end{itemize}
\end{lemma}

\medskip

Recalling that $W^u (p)\subset \hat\Lambda$, the first property of $Y$ applied to $p \in \hat\Lambda$ implies that
$$W^s_{loc}(Y)\subset \hat\Lambda.$$
The second property informs that
$$\bigcap_{i\,\in\,\mathbb{N}}\,g^i\left(W^s_{loc}(Y)\right)=Y=W^s_{loc}(Y).$$
Yet, $W^s_{loc}(Y)$ is open and $Y$ is closed so, owing to the connectedness of $M$, we must have $Y=M$. Consequently,
$$M=Y=\overline{W^u(p)}\subset \hat\Lambda$$
and so $\hat\Lambda=M.$
\end{proof}

\medskip

\begin{remark}\label{hyp_invariant}
If $\Lambda$ is a compact hyperbolic set for $g$ such that $\mu(\Lambda)>0$ and $\mu(\Lambda\cap Z)=0$, then, as $\Lambda$ and $\mu$ are $g$-invariant,
\begin{eqnarray*}
\mu\left(\Lambda\cap (M\backslash Z)\,\cap \,\cup_{j\,\in\, \mathbb{Z}}\,g^{-j}(Z)\right) &=&\mu\left(\Lambda\cap (M\backslash Z)\,\cap \,\cup_{j\,\in\, \mathbb{Z}}\,g^{-j}(\Lambda \cap Z)\right)\\
&=& \mu\left(\cup_{j\,\in\, \mathbb{Z}}\, \Lambda\cap (M\backslash Z)\,\cap \,g^{-j}(\Lambda \cap Z)\right)\\
&\leq& \sum_{j\in \mathbb{Z}}\, \mu\left(g^{-j}(\Lambda \cap Z)\right)=0
\end{eqnarray*}
which means that the iterates by $g$ of $\mu$ almost every point in $\Lambda \cap (M\backslash Z)$ remain there.
\end{remark}

\section{Proof of Theorem A}\label{BM}

Consider $f\in \text{Diff}^{~1}_{\mu, R}(M)$. If $f$ is Anosov or its integrated Lyapunov exponent (see Section~\ref{ILExp}) is zero, the proof ends. For instance, if $f=R$, then all orbits of $f$ have zero Lyapunov exponents. Otherwise, start approaching $f$ by $f_1$ of the open and dense set $\mathscr{D}_1$. Then, given $\epsilon > 0$, by Proposition~\ref{ZL_theorem} there exist a subset $Z \subset M$, whose complement in $M$ has Lebesgue measure smaller than $\epsilon$, and a diffeomorphism $f_2 \in \mathscr{D}_1$ which is $C^1$-close to $f_1$ (thus close to $f$) and is of class $C^2$ in an open neighborhood of  $Z$. Using Corollary~\ref{E} and Corollary~\ref{D2}, we then find a diffeomorphism $F \in \mathscr{D}$ whose set of periodic points is countable (so it has null Lebesgue measure), is $C^1$ close to $f_2$ (hence close to $f$) and is still $C^2$ when restricted to $Z$. According to Proposition~\ref{BV para reversiveis} applied to $F$, any compact hyperbolic set for $F$ has zero Lebesgue measure in $Z$, so its Lebesgue measure is smaller than $\epsilon$. The following result is the final key step to finish the proof.

\begin{proposition}\label{mainP(b)}
Given $\delta>0$, there is $g\in \text{Diff}^{~1}_{\mu, R}(M)$ which is $C^1$-close to $F$ and satisfies $\mathscr{L}(g) < \epsilon + \delta$.
\end{proposition}

Denote by $\mathscr{A}$ the $C^1$-open subset of $\text{Diff}^{~1}_{\mu, R}(M)$ of the $R$-reversible Anosov diffeomorphisms and, for any $k,n\in \mathbb{N}$, the set $$\mathscr{A}_{k,n}:=\left\{h\in\text{Diff}^{~1}_{\mu,R}(M)\colon \mathscr{L}(h)<\frac{1}{k} + \frac{1}{n}\right\}.$$
Clearly (see Section~\ref{ILExp}), the set
$$\mathscr{A}\cup \mathscr{A}_{k,n}$$
is $C^1$-open in $\text{Diff}^{~1}_{\mu, R}(M)$. After Proposition~\ref{mainP(b)}, we know that it is dense as well. Therefore, the set
$$\mathscr{A} \cup \left\{h\in\text{Diff}^{~1}_{\mu,R}(M)\colon \mathscr{L}(h)=0\right\}$$
is a countable intersection of the $C^1$ open and dense sets
$$\mathscr{A} \cup \left\{h\in\text{Diff}^{~1}_{\mu,R}(M)\colon \mathscr{L}(h)<\frac{1}{k} + \frac{1}{n}\right\}$$
and so it is residual.

\section{Proof of Proposition~\ref{mainP(b)}}\label{mainproposition}

Let $F\in \mathscr{D}$ be the diffeomorphism just constructed after fixing $f\in \text{Diff}^{~1}_{\mu, R}(M)$ and $\epsilon>0$. Recall that $F$ belongs to $\text{Diff}^{~1}_{\mu, R}(M) \backslash \mathscr{A}$, Lebesgue almost all its orbits are $(R,F)$-free, its set of periodic points has Lebesgue measure zero and any of its compact hyperbolic sets has Lebesgue measure smaller than $\epsilon$.

\subsection{Reducing locally the Lyapunov exponent}

The prior ingredient to prove Proposition~\ref{mainP(b)} is the next lemma whose statement is the reversible version of the Main Lemma in \cite{Bo}.

\begin{lemma}\label{mainlemma1}
Given $\eta,\, \delta>0$ and $\kappa\in\, ]0,1[$, there exists a measurable function $\mathcal{N}\colon M\rightarrow\mathbb{N}$ such that, for $x$ in a set $\hat{Z}$ with Lebesgue measure bigger than $1-\epsilon$ and every $n\geq \mathcal{N}(x)$, there exists $\varrho=\varrho(x,n)>0$ such that, for any ball $B(x,r)$, with $0<r<\varrho$, we may find $G\in \text{Diff}^{~1}_{\mu, R}(M)$, which is $\eta$-$C^1$-close to $F$, and compact sets $K_1\subset B(x,r)$ and $K_2\subset R\,F^{n}(K_1)\subset R\,F^{n}(B(x,r))$ satisfying:
\begin{enumerate}
\item [(a)] $F=G$ outside $\left(\bigcup_{j=0}^{n-1}\,F^j(\overline{B(x,r)})\right)\,\bigcup\,\left(\bigcup_{j=1}^{n}\,R\,F^j(\overline{B(x,r)})\right)$.\\

\item[(b)] For $j\in\{0,1,...,n-1\}$, the iterates $F^j(\overline{B(x,r)})$ and $R\,F^{j+1}(\overline{B(x,r)})$ are pairwise disjoint.\\

\item [(c)] $\mu(K_1)>\kappa\,\mu(B(x,r))$ and $\mu(K_2)>\kappa\,\mu(R\,F^{n}(B(x,r)))$.\\

\item [(d)] If $y_1\in{K}_1$ and $y_2\in{K_2}$, then $\frac{1}{n}\,\log\|DG_{y_i}^n\|<\delta\,\,$ for $i=1,2$.
\end{enumerate}
\end{lemma}

\medskip

Although the proof of this lemma follows closely the argument of \cite{Bo}, it is worth re\-gistering the fundamental differences between the previous result and \cite[Main Lemma]{Bo}. Firstly, each time we perturb the map $F$ around $F^{j}(x)$, for $j\in\{0,...,n-1\}$, we must ba\-lance with a perturbation around $R\,F^{j+1}(x)$ to prevent the perturbed diffeomorphism's exit from $\text{Diff}^{~1}_{\mu, R}(M)$. Thus the perturbations in $\bigcup_{j=0}^{n-1}\,F^j(\overline{B(x,r)})$ spread to a deformation of $F$ in $\bigcup_{j=1}^{n}\,R\, F^j(\overline{B(x,r)})$. This is possible because $F\in \mathscr{D}$, but our choice of $\varrho$ must be more judicious and, in general, smaller than the one in \cite{Bo} to avoid inconvenient intersections. Secondly, we need an additional control on the function $\mathcal{N}$ and on $\mu(K_2)$ to localize the computation of the Lyapunov exponents along the orbits of elements of $K_2$.

Aside from this, we also have a loss in measure. As $F$ is not globally $C^2$, instead of a function $\mathcal{N}\colon M\rightarrow\mathbb{N}$ with nice properties on a full measure set, during the proof \cite{Bo} we have to take out a compact hyperbolic component with, perhaps, positive measure. Fortunately, that portion has measure smaller than $\epsilon$, though its effect shows up in several computations and cannot be crossed off the final expression of the integrated Lyapunov exponent.

Regardless of these hindrances, reversibility also relieves our task here and there. For instance, the inequality for $y_2 \in K_2$ in the previous lemma, that is, $\|Dg_{y_2}^n\|<e^{n\delta}$, follows from the corresponding one for $y_1$ due to the reversibility and the fact that $\|A\|=\|A^{-1}\|$ for any $A\in \SL(2,\mathbb{R})$. Indeed, given $y_2\in K_2$, there exists $y_1\in K_1$ such that $y_2=R(F^{n}(y_1))=F^{-n}(R(y_1))$. Then (see Lemma~\ref{local})
$$\|DG_{y_2}^{n}\|=\|D(R\,G^{-n}\,R)_(y_2)\|\leq \|DG_{R(y_2)}^{-n}\|=\|DG_{y_1}^{n}\|<e^{n\delta}.$$

In what follows we will check where differences start emerging and summarize the essential lemmas where reversibility steps in.

\subsubsection{Sending $E^u$ to $E^s$}

\begin{definition}\cite[\S 3.1]{Bo}\label{rs}
Given $\eta>0$, $\kappa\in\, ]0,1[$, $n\in\mathbb{N}$ and $x\in M$, a finite family of linear maps $L_{j}:T_{F^{j}(x)}M\rightarrow{T_{F^{j+1}(x)}}M$, for $j=0,...,n-1$, is an $(\eta,\kappa)$-realizable sequence of length $n$ at $x$ if, for all $\gamma>0$, there is $\rho>0$ such that, for $j\in\{0,1,...,n-1\}$, the iterates $F^j(B(x,\rho))$ and $R(F^j(B(x,\rho)))$ are pairwise disjoint and, for any open non-empty set $U\subseteq{B(x,\rho)}$, there exist
\begin{enumerate}
\item [(a)] a measurable set $K_1\subseteq{U}$ such that $\mu(K_1)>\kappa\,\mu(U)$
\item [(b)] $h\in  \text{Diff}^{~1}_{\mu, R}(M)$, $\eta$-$C^{1}$-close to $F$ satisfying:
\begin{enumerate}
\item [(b.1)] $F=h$ outside $\left(\bigcup_{j=0}^{n-1}\, F^j(\overline{U})\right)\bigcup\left(\bigcup_{j=1}^{n}\,R(F^j(\overline{U}))\right)$
\item[(b.2)] if $y_1 \in K_1$, then $\|Dh_{h^{j}(y_1)}-L_{j}\|<\gamma\,\,$ for $j=0,1,...,n-1$.
\end{enumerate}
\end{enumerate}
\end{definition}

\medskip

Notice that, if the orbit of $x$ is $(R,F)$-free and not periodic (or periodic but with period greater than $n$) and we define $K_2:=R(F^n(K_1))$ and, for $j\in\{0,1,...,n-1\}$, the sequence
$$\begin{array}{cccc}
\tilde L_{j}\colon & T_{R(F^{n-j}(x))}M & \longrightarrow & {T_{R(F^{n-j-1}(x))}}M \\
& v & \longmapsto & DR_{F^{n-j-1}(x)} L_{n-j-1}^{-1} DR_{R(F^{n-j}(x))}(v)
\end{array}$$
then we obtain, for $y_2\in{K_2}$ and $j=0,1,...,n-1$, the inequality $\|Dh_{h^{j}(y_2)}-\tilde L_{j}\|<\gamma$.

\medskip

The following lemma is an elementary tool to interchange bundles using rotations of the Oseledets directions, and thereby construct realizable sequences. If $x \in M$ and $\theta \in \RR$, consider a local chart at $x$, $\varphi_x: V_x \rightarrow \RR^2$ and the maps ${D\varphi_x^{-1}}  \mathfrak{R}_\theta D\varphi_x: \RR^2 \rightarrow \RR^2$, where $\mathfrak{R}_\theta$ is the standard rotation of angle $\theta$ at $\varphi_x(x)$. Denote by $\mathcal{Y}$ the generic set given by Corollary~\ref{D2} and Corollary~\ref{E}, whose points have $(R,F)$-free and non-periodic orbits.

\begin{lemma}\cite[Lemma 3.3]{Bo}\label{basic}
Given $\eta>0$ and $\kappa\in\, ]0,1[$, there is $\theta_0 >0$ such that, if $x\in \mathcal{Y}$ and $|\theta|<\theta_0$, then $\{DF_x \mathfrak{R}_\theta\}$ and $\{\mathfrak{R}_\theta DF_x\}$ are $(\eta,\kappa)$-realizable sequence of length $1$ at $x$.
\end{lemma}

Now, the next result enables us to construct realizable sequences with a purpose: to send expanding Oseledets directions into contracting ones. This will be done at a region of $M$ without uniform hyperbolicity because there the Oseledets directions can be blended. More precisely, for $x \in \mathscr{O}^+(F)$ and $m\in \mathbb{N}$, let
$$\Delta_m (F,x)= \frac{\|DF_x^m|_{E^s(x)}\|}{\|DF_x^m|_{E^u(x)}\|}$$
and
$$\Gamma_m(F)=\left\{x \in \mathscr{O}^+(F)\cap \mathcal{Y}:\Delta_m (F,x)\geq \frac{1}{2}\right\}.$$

\medskip

\begin{lemma}\cite[Lemma 3.8]{Bo}\label{rotation}
Take $\eta>0$ and $\kappa\in\,]0,1[$. There is $m\in\mathbb{N}$ such that, for every $x\in\Gamma_m(F)$, there exists an $(\eta,\kappa)$-realizable sequence $\{L_0,L_1,...,L_{m-1}\}$ at $x$ with length $m$ satisfying
$$L_{m-1} (\ldots) L_1 L_0(E^u_x)=E^s_{F^m(x)}$$
and, consequently,
$$\tilde{L}_{m-1} (\ldots) \tilde{L}_1 \tilde{L}_0(E^u_{R(F^{m}(x))})=E^s_{R(x)}.$$
\end{lemma}

\medskip

The coming step is to verify that the above construction may be done in such a way that the composition of realizable sequences has small norm. Consider the $F$-invariant set
$$\Omega_m(F)=\bigcup_{n\, \in\, \mathbb{Z}}\, F^n(\Gamma_m(F)).$$
Then $\mathcal{H}_m=\mathscr{O}^+(F)-\Omega_m(F)$ is empty or its closure is a compact hyperbolic set \cite[Lemma 3.11]{Bo}. According to Proposition~\ref{BV para reversiveis}, $\mu(\mathcal{H}_m)<\epsilon$. Hence,

\begin{lemma}\cite[Lemma 3.13]{Bo}
Consider $\eta>0$, $\kappa\in\,]0,1[$ and $\delta>0$. There exists a measurable function $\mathcal{N}\colon M\rightarrow\mathbb{N}$ such that, for $x$ in a subset with Lebesgue measure greater that $1-\epsilon$ and all $n\geq N(x)$, we may find a $(\eta,\kappa)$-realizable sequence $\{L_j\}_{j=0}^{n-1}$ of length $n$ such that
$$\|L_{n-1}(\ldots)  L_0\|<e^{\frac{4}{5}\,n\,\delta}.$$
\end{lemma}

\medskip

If $\gamma$ is chosen small enough in the Definition~\ref{rs}, Lemma~\ref{mainlemma1} is a direct consequence of the preceding one.

\subsection{Reducing globally the Lyapunov exponent}

After Lemma~\ref{mainlemma1} we know how to find large values of $n$ such that, for some perturbation $G\in\text{Diff}^{~1}_{\mu, R}(M)$ of $F$, we get $\|DG^n_x\|<e^{n\delta}$ for a considerable amount of points $x$ inside a small ball and its image by $RF$. However, the Lyapunov exponent is an asymptotic concept and we need to evaluate, or find a good approximation of it on a set with full $\mu$ measure. In this section we will extend the local procedure to an almost global perturbation, which allows us to draw later on global conclusions. The classic ergodic theoretical construction of a Kakutani castle \cite{Kakutani} is the bridge between these two approaches, as was discovered in \cite[\S 4]{Bo}. The main novelty here is that, when building some tower of the castle, we simultaneously built its mirror inverted reversible copy.

\subsubsection{A reversible Kakutani castle}

Let $A\subseteq M$ be a borelian subset of $M$ with positive Lebesgue measure and $n\in \mathbb{N}$. The union of the mutually disjoint subsets $\bigcup_{i=0}^{n-1} F^i(A)$ is called a \emph{tower}, $n$ its \emph{height} and $A$ its \emph{base}. The union of pairwise disjoint towers is called a \emph{castle}. The \emph{base of the castle} is the union of the bases of its towers. The first return map to $A$, say $\tau: A \rightarrow \mathbb{N}\cup \{\infty\}$, is defined as $\tau(x)=\inf \{ n\in \mathbb{N}: F^n(x) \in A\}$. Since $\mu(A)>0$ and $F$ is measure-preserving, by Poincar\'e recurrence theorem the orbit of Lebesgue-almost all points in $A$ will come back to $A$. Thus, $\tau(x) \in \mathbb{N}$ for Lesbesgue almost every $x \in A$. If $A_n=\{x \in A: \tau(x) =n\}$, then $\mathcal{T}_n= A_n \cup F(A_n) \cup \ldots \cup F^{n-1}(A_n)$ is a tower, $\bigcup_{n\,\in\, \mathbb{Z}} F^n(A)$ is $F$-invariant and it is the union of the towers $\mathcal{T}_n$: it is a castle with base $A$. Moreover,

\begin{lemma}\cite[pp. 70 and 71]{Halmos} \label{Auxiliar2} For every borelian $U$ such that $\mu(U)>0$ and every $n\in\mathbb{N}$, there exists a positive measure set $V\subset U$ such that $V$, $F(V), \ldots, F^{n}(V)$ are pairwise disjoint. Besides, $V$ can be chosen in such a way that no set that includes $V$ and has larger Lebesgue measure than $V$ has this property.
\end{lemma}

\medskip

Fix $\eta, \,\delta>0$ and take $0<\kappa < 1$ such that $1-\kappa < \delta^2$. Apply Lemma~\ref{mainlemma1} to get a function $\mathcal{N}$ as stated. For each $n \in \mathbb{N}$, consider $P_n=\{x \in M: \mathcal{N}(x) \leq n\}$. Clearly, $\lim_{n \in \mathbb{N}} \mu(P_n)\geq 1-\epsilon$. So there is $\alpha \in \mathbb{N}$ such that $\mu(M\backslash P_{\alpha}) < \epsilon + \delta^2$, and therefore $\mu\left(M\backslash (P_{\alpha}\cup R(P_\alpha))\right)< \epsilon +  \delta^2$. For $U:=P_{\alpha}\cup R(P_\alpha)$ and $\alpha$, Lemma~\ref{Auxiliar2} gives a maximal set $\mathcal{B}\subset P_\alpha\cup R(P_\alpha)$ with positive Lebesgue measure such that $\mathcal{B}$, $F(\mathcal{B}), \ldots, F^{\alpha}(\mathcal{B})$ are mutually disjoint. Then the set $\hat{\mathcal{Q}}=\cup_{n\, \in\, \mathbb{Z}} F^n(\mathcal{B})$ is the Kakutani castle associated to the base $\mathcal{B}$. Observe that, by the maximality of $\mathcal{B}$, the set $\hat{\mathcal{Q}}$ contains $P_{\alpha}\cup R(P_\alpha)$, and so $\mu\left(\hat{\mathcal{Q}}\backslash (P_{\alpha}\cup R(P_\alpha))\right)<\epsilon + \delta^2$.

Consider now the castle ${\mathcal{Q}}\subset \hat{\mathcal{Q}}$ whose towers have heights less that $3\alpha$. Adapting the argument in \cite[Lemma 4.2]{Bo}, we obtain:

\begin{lemma} $\,\,\,\mu\left(\hat{\mathcal{Q}}\backslash {\mathcal{Q}}\right)< 3(\epsilon + \delta^2).$
\end{lemma}

Furthermore,

\begin{lemma} \label{claim1}$\,$
\begin{itemize}
\item[(a)] $\mu\left(\mathcal{B}\triangle R(\mathcal{B})\right)=0$.
\item[(b)] If $\mathcal{T}$ is a tower of height $n$, then also is $R(\mathcal{T})$. Moreover, $RF^{n}(\mathcal{T}\cap \mathcal{B})=R(\mathcal{T})\cap \mathcal{B}$.
\end{itemize}
\end{lemma}

\begin{proof}
\noindent (a) We will show that $R(\mathcal{B})\subset \mathcal{B} \,\, \text{ modulo } \mu$. Assume that there exists a positive $\mu$-measure subset $C\subset R(\mathcal{B})$ such that $C$ is not contained in $\mathcal{B}$. Observe that $C\subset P_{\alpha}\cup R(P_\alpha)$ because $P_\alpha\cup R(P_\alpha)$ is $R$-invariant and $\mathcal{B}\subset P_\alpha\cup R(P_\alpha)$. As $\mathcal{B}$ is maximal and there are points of $C$ out of $\mathcal{B}$, we have $F^i(C)\cap F^j(C)\not = \emptyset$ for some $i\neq j\in\{0,...,\alpha\}$. However, $R(C)\subset \mathcal{B}$ and $\mu(R(C))=\mu(C)>0$, so $F^i(R(C))\cap F^j(R(C))=\emptyset$ which, using reversibility, is equivalent to $R(F^{-i}(C))\cap R(F^{-j}(C))=\emptyset$, that is, $F^{-i}(C)\cap F^{-j}(C)=\emptyset$, a contradiction.
\medskip

\noindent (b) This is a direct consequence of (a). Since $\mathcal{T}$ is a tower of height $n$, its first floor $T_0$ and its top floor $T_n$ are in $\mathcal{B}$. By (a), $R(T_0)$ and $R(T_n)$ are in $\mathcal{B}$ as well, and so they are, respectively, the top and first floor of the tower $R(\mathcal{T})$, and its height has to be $n$ too.
\end{proof}

At this stage, we may ask about the effect of the existence of a hyperbolic set $\Lambda\cap (M\backslash Z)$ with positive, although small, Lebesgue measure. Could a typical orbit $x\in \mathcal{B}$ visit regions with hyperbolic-type behavior and positive measure? In fact, the reported situation almost never happens due to Remark~\ref{hyp_invariant}: only a null Lebesgue measure set of points in $\mathcal{B}$ may visit $M \backslash Z$.

\subsubsection{Regular families of sets}
Following \cite{Lebesgue}, we say that colletion $\mathcal{V}$ of mensurable subsets of $M$ is a \emph{regular family} for the Lebesgue measure $\mu$ if there exists $\nu>0$ such that $diam(V)^2 \leq \nu \mu(V)$ for all $V \in \mathcal{V}$, where $diam(A)=\sup\{d(x,y), x,y \in A\}$. In what follows, we will prove that the family of all ellipses with controled eccentricity constitutes a regular family for the Lebesgue measure.

An ellipse $E \subset M$ whose major and minor axes have lengths $a$ and $b$, respectively, has eccentricity $e\geq 1$ if it is the image of the unitary disk $D \subset M$ under $\Phi \in \SL(2,\mathbb{R})$ and $\|\Phi\|= e = \sqrt{a/b}$. Given $e_0>1$, the family of all ellipses whose eccentricity stays between $1$ and $e_0$ is a regular family for the Lebesgue measure (just take $\nu= e_0^2$).

Let $\mathcal{B}$ be the base of the castle $\mathcal{Q}$ and let $n(x)$ be the height of the tower containing $x$. Recall that we have $\mathcal{N}(x) \leq \alpha \leq n(x)$.

\begin{lemma}\label{elipses2}
Consider the castle $\mathcal{Q}$ and $x \in \mathcal{B}$. There exists $r(x)>0$ and a ball $B(x,r(x))$ such that the set $B(x,r(x)) \cup R (F^{n(x)}(B(x,r(x))))$ is a regular family.
\end{lemma}

\begin{proof}
Clearly, the sets $B(x,r(x))$ are regular (choose $\nu=4/\pi$). Let us see that $R\,F^{n(x)}(B(x,r(x)))$ is also regular. Notice that, in general, this set is not an ellipse. However, if $B(x,r(x))$ is small, then $R\,F^n(B(x,r(x)))$ is close to its first order approximation, that is $DR\,DF^n(B(x,r(x)))$, which is an ellipse.

First observe that the height of a tower is constant in balls centered at points of $\mathcal{B}$ with sufficiently small radius \cite[Section 4.3]{Bo}. Denote by $C_F:=\max_{z\in M}\|DF_z\|$. Since $\mu$ is $F$ and $R$ invariant, if $r(x)<1$ we have
\begin{eqnarray*}
\left[diam(R\, F^{n(x)}(B(x,r(x))))\right]^2 &=& \left[diam(F^{n(x)}(B(x,r(x))))\right]^2 \leq (2\,r(x)\,C_F)^{2\,n(x)}\\
&=& \frac{(2\,C_F)^{2\,n(x)}\,r(x)^{2\,n(x)-2}}{\pi}\,\pi\, r(x)^2 \\
&\leq& \frac{(2\,C_F)^{6\alpha}\,r(x)^{6\alpha-2}}{\pi}\,\pi\, r(x)^2\\
&\leq& \frac{(2\,C_F)^{6\alpha}}{\pi}\,\pi\, r(x)^2\\
&=& \frac{(2\,C_F)^{6\alpha}}{\pi}\,\mu(B(x,r(x)) \\
&=& \nu\, \mu\left(R\, F^{n(x)}(B(x,r(x)))\right)
\end{eqnarray*}
where $\nu=\frac{(2\,C_F)^{6\alpha}}{\pi}$.
\end{proof}

\medskip

\subsubsection{Construction of $g$}

The last auxiliary result says that it is possible, using Vitali Covering Lemma and Lemma~\ref{elipses2}, to cover the base $\mathcal{B}$ essentially with balls and ellipses.

\begin{lemma}\cite[\S4.3]{Bo}\label{covering_lemma1}
Let $\gamma>0$ satisfy $\gamma<\delta^2\alpha^{-1}$. Then:
\begin{itemize}
\item[(a)] There is a compact castle $\mathcal{Q}_1$ contained in $\mathcal{Q}$ and an open castle $\mathcal{Q}_2$ containing $\mathcal{Q}$ with the same shape\footnote{This means that the castles have the same number of towers and the towers have the same heights.} as $\mathcal{Q}$ and such that $\mu(\mathcal{Q}_2\backslash \mathcal{Q}_1)<\gamma$.
\item[(b)] The base $\mathcal{B}_3$ of the castle $\mathcal{Q}_2\cap\mathcal{Q}$ may be covered by a finite number of balls $B(x_i,r'(x_i))$ and their images  $R\,F^{n_i}(B(x_i,r'(x_i))$, where $x_i \in \mathcal{B}_3$ and $r'(x_i)$ is small enough so that $n(x)_{|_{B(x_i,r'(x_i))}}\equiv n_i$ and
$$\frac{\mu\left(\mathcal{B}_3\,\backslash\, \overset{\cdot}{\bigcup}\, B(x_i,r(x_i))\,\cup\, R\,F^{n_i}(B(x_i,r(x_i)))\right)}{\mu(\mathcal{B}_3)}<\gamma.$$
\end{itemize}
\end{lemma}

\medskip

Once the covering $\bigcup B(x_i,r(x_i))\cup R\,F^{n_i}(B(x_i,r(x_i)))$ is found, Lemma~\ref{mainlemma1} provides, for each $i$, a diffeomorphism $g_i\in \text{Diff}^{~1}_{\mu, R}(M)$ which is $C^1$-close to $F$ and compact sets
$$K_1^i\subset B(x_i,r(x_i))\,\,\,\, \text{ and } \,\,\,\, K_2^i\subset R\,F^{n_i}(B(x_i,r(x_i))))$$
such that:
\begin{enumerate}
\item $g_i=F$ outside $[\bigcup_{j=0}^{n_i-1}\,F^j(\overline{B(x_i,r(x_i))})]\,\bigcup\,[\bigcup_{j=1}^{n_i}\,R(F^j(\overline{B(x_i,r(x_i))}))]$.\\
\item For $j\in\{0,1,...,n_i-1\}$, the iterates $F^j(\overline{B(x_i,r(x_i))})$ and $R(F^{j+1}(\overline{B(x_i,r(x_i))}))$ are pairwise disjoint.\\
\item $\mu(K_1^i)>\kappa\,\mu(B(x_i,r(x_i)))$ and $\mu(K_2^i)>\kappa\,\mu(R\,F^{n_i}(B(x_i,r(x_i))))$.\\
\item If $y_1\in{K}_1^i$ and $y_2\in{K_2^i}$, then $\log\|(Dg_i^{n_i})_{y_1}\|<n_i\,\delta$ and $\log\|(Dg_i^{n_i})_{y_2}\|<n_i\,\delta$.
\end{enumerate}

\medskip

Finally, we define the diffeomorphism $g\in \text{Diff}^{~1}_{\mu, R}(M)$ by $g=g_i$ in each component $$[\bigcup_{j=0}^{n_i-1}\,F^j(\overline{B(x_i,r(x_i))})]\,\bigcup\,[\bigcup_{j=1}^{n_i}\,R\,F^j(\overline{B(x_i,r(x_i))}))]$$
and $g=f$ elsewhere.\\

\subsubsection{Estimation of $\mathscr{L}(g)$}

For $\varphi \in \text{Diff}^{~1}(M)$, let $C_{\varphi}=\max\,\, \{\|D\varphi_z\|: z \in M\}$ and denote by $C_1$ the maximum of the set
$$\left\{\,C(\varphi): \varphi \in \text{Diff}^{~1}_{\mu, R}(M) \,\text{ and $\varphi$ is }\, \text{$\eta$-$C^1$-close to } F\,\right\}.$$
As in \cite{Bo}, despite the necessary adjustments, there are a constant $C_2>0$, a positive integer $N\geq \delta^{-1}\,\alpha$, a $g$-castle $K$ of the same type as $\mathcal{Q}_2$ and a subset $\mathcal{G}=\bigcap_{j=1}^{N-1}\,g^{-j}(K)$ of $M$ such that
\begin{eqnarray*}
\mathscr{L}(g) &=& \int_{\mathcal{G}}\, \lambda^+(g)  \,d\mu + \int_{\hat{Z}\backslash\mathcal{G}}\, \lambda^+(g)  \,d\mu + \int_{M\backslash \hat{Z}}\, \lambda^+(g)  \,d\mu \\
&\leq& \int_{\mathcal{G}}\, \,\frac{1}{N} \,\log \|Dg^N\| \,d\mu + \int_{\hat{Z}\backslash \mathcal{G}}\, \lambda^+(g)  \,d\mu + \int_{M\backslash \hat{Z}}\, \lambda^+(g)  \,d\mu\\
&\leq& C_2\, \delta + \ln\,(C_1) (\delta + \epsilon) + \int_{M\backslash \hat{Z}}\, \lim_{n \rightarrow +\infty}\,\frac{1}{n}\,\ln\,\|Dg^n_x\| \,d\mu \\
&\leq& C_2\, \delta + \ln\,(C_1)(\delta + \epsilon) + \ln\,(C_1)\,\epsilon \\
&=& (C_2+\ln\,(C_1))\, \delta + 2\ln\,(C_1)\,\epsilon.
\end{eqnarray*}

\bigskip

\section*{Acknowledgements}
The authors are grateful to Ant\'onio Machiavelo for enlightening discussions. M\'ario Bessa was partially supported by National Funds through FCT (Funda\c{c}\~{a}o para a Ci\^{e}ncia e a Tecnologia) project PEst-OE/MAT/UI0212/2011. CMUP has been funded by the European Regional Development Fund, through the programme COMPETE,
and by the Portuguese Government through the FCT project PEst-C/MAT/UI0144/2011. Alexandre Rodrigues has benefited from the FCT grant SFRH/BPD/84709/2012.

\bigskip
\bigskip
\bigskip
\bigskip

\end{document}